\documentclass[11pt, leqno]{amsart}
\usepackage[T1]{fontenc}

\usepackage{wrapfig}

\usepackage{stmaryrd}

\usepackage{indentfirst}
\usepackage{amstext}
\usepackage{amsopn}
\usepackage{amsfonts}
\usepackage{amsmath}
\usepackage{latexsym}
\usepackage{amscd}
\usepackage{amssymb}
\usepackage{amsmath}
\usepackage{leftidx}

\usepackage{graphicx}

\usepackage{colortbl}

\usepackage{rotating}
=\oddsidemargin \font\teneufm=eufm10 \font\seveneufm=eufm7
\font\fiveeufm=eufm5
\newfam\eufmfam
\textfont\eufmfam=\teneufm \scriptfont\eufmfam=\seveneufm
\scriptscriptfont\eufmfam=\fiveeufm

\def\1{\mbox{\bf 1}}

\newtheorem{stheorem}{Theorem}[section]
\newtheorem{question}[stheorem]{Question}
\newtheorem{déf}[stheorem]{Définition}

\newtheorem{rap}[stheorem]{Rappel}
\newtheorem{prop}[stheorem]{Proposition}
\newtheorem{lem}[stheorem]{Lemme}
\newtheorem{coro}[stheorem]{Corollaire}
\newtheorem{thm}[stheorem]{Théorème}

\usepackage[numbers]{natbib}
\usepackage[colorlinks=true,pagebackref=false,linktocpage=true]{hyperref}

\usepackage{tikz}
\usepackage{tikz-cd}
\usepackage[french]{babel}

\usepackage{adjustbox}
\usepackage{enumitem}
\usepackage[scr=rsfs]{mathalpha}
\usepackage{mathtools}

\usepackage{dynkin-diagrams}
\usepackage{musicography}

\theoremstyle{remark}

\newtheorem{rmq}[stheorem]{Remarque}

\newcommand{\GG}{\mathbb{G}}

\newcommand{\Ker}{\mathrm{Ker}}

\newcommand{\Spec}{\mathrm{Spec}}
\newcommand{\Specm}{\mathrm{Specm}}
\newcommand{\Aut}{\mathrm{Aut}}
\newcommand{\Gal}{\mathrm{Gal}}

\newcommand{\Kh}{\widehat{K}}

\newcommand{\Rh}{\widehat{R}}

\newcommand{\Kt}{\widetilde{K}}
\newcommand{\Ktnr}{\Kt^{\mathrm{n.r.}}}
\newcommand{\Rt}{\widetilde{R}}
\newcommand{\Rtnr}{\Rt^{\mathrm{n.r.}}}
\newcommand{\Galnr}{\Gamma^\mathrm{n.r.}}

\newcommand{\Gc}{\mathcal{G}}

\newcommand{\Oc}{\mathcal{O}}
\newcommand{\mm}{\mathfrak{m}}
\newcommand{\Xf}{\mathfrak{X}} 

\newcommand{\Ac}{\mathcal{A}} 
\newcommand{\Cc}{\mathcal{C}}

\newcommand{\ImmBT}{\mathscr{B}}

\defcitealias{stacks-project}{Stacks}
\defcitealias{SGA3}{SGA3}

\newcommand{\stacks}[1]{\citepalias[\href{https://stacks.math.columbia.edu/tag/#1}{Tag #1}]{stacks-project}}
\newcommand{\sga}[1]{\citepalias[#1]{SGA3}}

\begin{document}

\title[Arithmétique des schémas en groupes de Bruhat-Tits (cas semi-local)]{Arithmétique des schémas en groupes de Bruhat-Tits sur un anneau de Dedekind semi-local}

\author[A.\ Zidani]{Anis Zidani}
\address{Anis Zidani\\
Christian-Albrechts-Universität zu Kiel\\
Mathematisches Seminar\\
Heinrich-Hecht-Platz 6\\
24118 Kiel, Deutschland\\
and Institute of Mathematics ”Simion Stoilow” of the Romanian Academy\\
21 Calea Grivitei Street\\
010702 Bucharest\\
Romania}

\date{\today}

\begin{abstract} 
L'objectif de cet article est de poser les bases de l'étude cohomologique des schémas en groupes de Bruhat-Tits sur un anneau de Dedekind semi-local. On obtient notamment une preuve simplifiée de la conjecture de Grothendieck-Serre dans ce cas de figure et également un résultat analogue pour les schémas en groupes de Bruhat-Tits d'un groupe semi-simple simplement connexe.
\end{abstract}

\maketitle

\bigskip
\bigskip

\noindent{\bf Mots clés :} Groupes algébriques, Groupes réductifs, Théorie de Bruhat-Tits, Conjecture de Grothendieck-Serre, Torseurs, Modèles entiers, Anneaux de Dedekind semi-locaux, Approximation faible.

\medskip


\bigskip

\tableofcontents

\section*{Introduction}

Le point de départ de cet article provient de la question posée par Eva Bayer-Fluckiger et Uriya A. First dans \cite{BFF} sur des objets qui généralisent les schémas en groupes de Bruhat-Tits sur des anneaux de Dedekind semi-locaux. 
\medskip

Considérons donc $R$, un anneau de Dedekind semi-local connexe de dimension $1$ et $K$ son corps de fractions. Par définition, un idéal maximal $\mm$ de $R$ définit par localisation un anneau de valuation discrète $R_{\mm}$, de complété noté $\Rh_{\mm}$. Notons également $\Kh_{\mm}$, le corps de fractions de $\Rh_{\mm}$.
\medskip

Introduisons la définition suivante :

\begin{déf}\label{defGroupesBT}
    Soit $G$ un groupe algébrique réductif sur $K$ et $\Gc$ un schéma en groupes lisse sur $R$ tel que $G:=\Gc_{K}$. On dit que $\Gc$ est un \textbf{schéma en groupes stabilisateur d'une facette} (resp. \textbf{schéma en groupes parahorique}) de $G$ si pour tout idéal maximal $\mm$ de $R$, le schéma en groupes $\Gc_{\Rh_{\mm}}$ est stabilisateur d'une facette dans l'immeuble de Bruhat-Tits $\ImmBT(G_{\Kh_{\mm}})$, cf. \cite[Définition 3.9.]{Article1} (resp. est parahorique, cf. \cite[Définition 6.2.]{Article1}).
\end{déf}

Notons que cette définition coïncide avec celle prise par Heinloth dans \cite{Heinloth}, dans le cas semi-simple, et dans le cas où la base est une courbe projective lisse sur un corps. 

Par ailleurs, dans le cas des tores, un schéma en groupes stabilisateur d'une facette peut correspondre au modèle de Néron du tore (sachant que l'immeuble d'un tore est réduit à un sommet). Notons d'ailleurs qu'il s'agit d'un exemple où le modèle considéré n'est pas nécessairement affine.
\medskip

La question de Bayer-Fluckiger et First sur les torseurs rationnellement triviaux s'énonce donc ainsi :

\begin{question}[{\cite[Question 6.4]{BFF}}] \label{QuestionBFF}
    Soit $\Gc$, un schéma en groupes sur $R$ tel que $G:=\Gc_{K}$ est réductif.
    Est-ce que le morphisme de changement de base :
    $$H^1_{\textrm{\upshape ét}}(R,\Gc) \rightarrow H^1_{\textrm{\upshape ét}}(K,G)$$
    est injectif lorsque $\Gc$ est :
    \begin{enumerate}
        \item un schéma en groupes stabilisateur d'une facette de $G$ ?
        \item un schéma en groupes parahorique de $G$ ?
    \end{enumerate}
\end{question}

Dans l'article, les auteurs supposent également que les corps résiduels de $R$ sont parfaits, mais précisent toutefois que cela est simplement une hypothèse simplificatrice.
\medskip

Dans \cite{BFFH}, les mêmes auteurs ont trouvé un contre-exemple dans le cas où le groupe $G$ est non connexe et de composante neutre adjointe. Ce contre-exemple est plus précisément construit en \cite[§4.]{BFFH}. Ceci les a conduit à formuler une conjecture plus faible dans le dernier paragraphe de \cite[§5.]{BFFH} : est-ce que la question (1) de \ref{QuestionBFF} est satisfaite lorsque la facette considérée est une chambre et $G$ est résiduellement quasi-déployé sur chaque $\Kh_\mm$ ? (cf. \cite[Définition 3.4.]{Article1}). Ceci était déjà connu de Bruhat et Tits dans le cas complet à corps résiduel parfait (cf. \cite[3.9. Lemme]{BT3}). On répond positivement à cette conjecture dans cet article. C'est l'objet du théorème \ref{ThmChambres}.
\medskip

Comme signalé dans \cite{Article1}, il s'avère qu'un contre-exemple où $G$ est connexe avait déjà été trouvé pour le cas $(1)$ de la question \ref{QuestionBFF} dans le cas d'un anneau de valuation discrète complet et d'un groupe adjoint quasi-déployé de type ${}^2A_3$ et déployé par une extension non ramifiée par Bruhat et Tits dans \cite[5.2.13.]{BT2}.

Dans l'article \cite{Article1}, on a alors généralisé ce contre-exemple et calculé tous les noyaux possibles dans le cas quasi-déployé et adjoint sur un corps valué hensélien (cf. \cite[Théorème 6.15.]{Article1}). On a également montré que le noyau du morphisme de la question \ref{QuestionBFF} dans le cas $(2)$ est trivial dans ce cas de figure. On se propose dans le présent article de généraliser ces résultats pour n'importe quel groupe $G$ adjoint sur $K$ et quasi-déployé sur chaque $\Kh_{\mm}$ (cf. les théorèmes \ref{ThmStabQSplit} et \ref{ThmParahoQSplit}).

Malgré nos efforts, le cas $(2)$ de la question \ref{QuestionBFF} est toujours une question ouverte.
\medskip

Notons également que cette question \ref{QuestionBFF} est une généralisation de la conjecture de \linebreak Grothendieck-Serre dans le cas d'un anneau de valuation discrète. En effet, il s'agit du cas où le schéma en groupes parahorique est associé à un sommet hyperspécial (dans ce cas, le schéma en groupes est réductif, cf. \cite[Lemme 5.1.]{Article1}).
\medskip

La première tentative de preuve de ce cas est due à Nisnevich dans sa thèse \cite{TheseNisnevich}. L'idée est d'utiliser les techniques de recollements pour montrer que le problème se ramène au cas complet et à un problème de décomposition.

Dans notre cas de figure, en reprenant les notations de la question \ref{QuestionBFF}, un problème de décomposition reviendrait à se demander si l'égalité suivante est satisfaite : $$\prod_{\mm} G(\Kh_\mm)=G(K)\,\prod_\mm \Gc(\Rh_\mm):=\left\{(g_Kg_\mm)_{\mm} \mid (g_K,(g_\mm)_\mm)\in G(K)\times \prod_\mm \Gc(\Rh_\mm)\right\}.$$
Le problème de décomposition de Nisnevich est alors le cas où $\Gc$ est supposé réductif.

Notons d'ailleurs qu'obtenir cette décomposition signifie également que le \textit{groupe de classes} (qui est a priori seulement un ensemble pointé) $c(\Gc):= \prod_\mm\Gc( \Rh_\mm)\backslash \prod_\mm G(\Kh_\mm)/G(K)$ est trivial. Cet objet a aussi été étudié par Nisnevich dans sa thèse (cf. \cite[Chapter I]{TheseNisnevich}).

Ensuite, dans la note \cite{Nisnevich}, Nisnevich apporte des améliorations à sa tentative et indique un résultat de Bruhat et Tits non encore publié à l'époque qui donne le cas semi-simple complet.

Ce résultat (et sa preuve) va ensuite être publié dans \cite[5.2.14. Proposition.]{BT2}, bien qu'il ne soit pas formulé de manière cohomologique. On a montré dans \cite{Article1} que c'est effectivement équivalent à l'énoncé du cas complet, et que le cas réductif peut être également obtenu en ajustant la preuve (cf. \cite[Proposition 5.4.]{Article1}).

Le cas des tores a été ensuite prouvé plus tard par Colliot-Thélène et Sansuc dans \cite[Theorem. 4.1.]{CTS} mais dans un contexte bien plus général. Il s'avère que dans notre contexte on peut en fournir une preuve bien plus simple dans le cas complet : cela est l'objet de \cite[Lemme 5.3.(ii)]{Article1}.

Enfin, Guo dans \cite{Gu} clarifie la preuve de Nisnevich tout en optant cette fois pour une autre preuve du cas complet en passant par une technique de réduction au cas anisotrope. Il ajoute également le cas où l'anneau est de plus semi-local.
\medskip

On propose également dans cet article d'obtenir une preuve simplifiée et nouvelle de ce résultat en obtenant une autre preuve du problème de décomposition, et en combinant cela au cas complet que l'on a déjà traité dans \cite[Proposition 5.4.]{Article1}.
\medskip

Notre objectif principal est donc de répondre de la manière la plus exhaustive possible à la question \ref{QuestionBFF}. Les corps résiduels de $R$ ne sont donc pas supposés parfaits (sauf mention explicite du contraire). 
\medskip

Notre stratégie reprend essentiellement celle de Nisnevich. On utilise les techniques de recollements pour découper le problème en deux : résoudre le cas complet (ce qui a déjà été exploré dans \cite{Article1}) et résoudre un problème de décomposition. Le fait de sortir du cas réductif nécessite toutefois d'utiliser de nouvelles méthodes (ou d'utiliser de manière plus astucieuse celles déjà connues). 
\medskip

Discutons désormais du problème de décomposition. La stratégie que l'on adopte dans ce papier utilise, pour tout $\mm$, le groupe $G(\Kh_\mm)^+$ engendré par les $\Kh_\mm$-points des sous-groupes de racines de $G_{\Kh_\mm}$. Elle repose sur le fait de montrer que $\prod_\mm G(\Kh_\mm)^+ \subset G(K)\,\prod_\mm \Gc(\Rh_\mm)$, ce qui simplifie grandement la problématique, car $\prod_\mm G(\Kh_\mm)^+$ est en pratique suffisamment gros pour conclure sur un certain nombre de cas.

Lorsque $G$ est $K$-isotrope, il était déjà connu dans la littérature que l'on pouvait approcher $\prod_\mm G(\Kh_\mm)^+$ avec des éléments de $G(K)^+$ (cf. \cite[Lemme 5.6.]{BourbakiGille}). Le cas où $G$ est \linebreak $K$-anisotrope et où il existe un $\mm$ tel que $G$ soit $\Kh_\mm$-isotrope est nettement plus délicat et n'a pas été étudié dans la littérature.

L'idée novatrice dans cet article réside alors en l'utilisation du théorème de Prasad  (cf. proposition \ref{Prasad}) pour montrer que $\prod_\mm G(\Kh_\mm)^+ \subset G(K)\,\prod_\mm \Gc(\Rh_\mm)$, et ce, même si $G$ est $K$-anisotrope : le problème de décomposition est ainsi simplifié dans tous les cas de figure.
\medskip

Le plan de cet article est le suivant :
\begin{enumerate}
    \item La première partie est dédiée aux techniques de recollements. On y généralise ce qui a déjà été fait par Nisnevich et Guo pour inclure le cas de schémas en groupes plus généraux non nécessairement affines (en particulier ceux qui nous intéressent).
    \item La seconde partie est dédiée aux techniques d'approximation. On y développe des résultats qui simplifient considérablement l'étude du problème de décomposition. 
    \item La troisième partie est dédiée à l'établissement de lemmes cruciaux et des principaux théorèmes de l'article.
\end{enumerate}
\medskip

On peut déjà annoncer tout de suite que dans le cas où les corps résiduels sont parfaits, et que le groupe $G$ est semi-simple simplement connexe, la question \ref{QuestionBFF} admet une réponse positive :

\begin{thm}\label{CasParfaitSSSC}
    Supposons que les corps résiduels de $R$ soient parfaits. Soit $G$ un groupe semi-simple simplement connexe.
    Alors les schémas en groupes stabilisateur d'une facette et parahoriques pour $G$ coïncident et lorsque $\Gc$ en est un, le morphisme de changement de base :
    $$H^1_{\textrm{\upshape ét}}(R,\Gc) \rightarrow H^1_{\textrm{\upshape ét}}(K,G)$$
    est injectif.
\end{thm}

Notons également que, bien que l'on ait une preuve essentiellement uniforme, les difficultés que l'on rencontre dans cet article (et plus généralement dans ce sujet), qui sont de nature aussi bien techniques que conceptuelles, sont dues au fait que la théorie de Bruhat-Tits a été peu examinée dans le cas d'un groupe sur un corps complet valué discrètement non quasi-déployé par une extension non ramifiée.
\medskip

\noindent{\bf Remerciements.}
${}$
\medskip

L'auteur remercie Philippe Gille et Ralf Köhl, pour leurs soutiens, leurs accompagnements et leur relecture du présent article. Il remercie également Ofer Gabber pour avoir répondu à la question en \ref{RmqIndQuasiAff}. 

L'auteur est également reconnaissant vis-à-vis de la Studienstiftung des deutschen Volkes pour avoir supporté financièrement ce projet. Il a été également soutenu par le projet “Group schemes, root systems, and related representations”, financé par l'Union Européenne - NextGenerationEU à travers le "Romania’s National Recovery and Resilience Plan" (PNRR) call no. PNRR-III-C9-2023-I8, Project CF159/31.07.2023, et coordonné par le Ministre de la Recherche, de l'Innovation et de la Numérisation (MCID) de Roumanie. 

\section*{Notations et conventions} 

Pour tout corps $k$, la notation $k^s$ désigne une clôture séparable de $k$.
\medskip

Rappelons que tout schéma $X$ localement de présentation finie séparé sur un schéma intègre $S$ de corps de fonctions $k$ est tel que $X(S)\rightarrow X(k)$ est injectif. Cette inclusion est implicite tout le long du document (cf. \cite[Corollary 9.9.]{GortzWedhorn}).
\medskip

Nous utilisons la définition de groupe réductif de Chevalley et Borel (cf. \cite{BorelLinAlgGroups}). En particulier, ils sont affines, lisses et connexes.
\medskip

Le premier ensemble non abélien de cohomologie étale et fppf considéré dans cet article est défini par Milne dans \cite[III.\S 4.]{MilneEtale} par le procédé de Čech. De manière équivalente, ils sont donnés par les classes d'isomorphismes de torseurs faisceautiques, et sont donc a priori non nécessairement représentables par des schémas (cf. \cite[III. Proposition 4.6.]{MilneEtale}).
\medskip

Dans toute la suite on considère $R$ un anneau de Dedekind semi-local connexe de dimension $1$ et $K$ son corps de fractions. Tout ce qui va suivre dans cet article se généralise trivialement au cas non connexe et au cas où une composante est de dimension $0$. 
\medskip

Pour tout idéal maximal $\mm$ dans $R$, notons que $R_{\mm}$ est un anneau de valuation discrète et qu'il munit donc $K$ d'une valuation discrète. On note $\Kh_{\mm}$ et $\Rh_{\mm}$ les complétés associés respectifs de $K$ et $R_{\mm}$. Notons également $R_{\mm}^h$, l'hensélisé de $R_{\mm}$ et $K_{\mm}^h$ son corps de fractions (cf. \stacks{0BSK}).
\medskip

Étant donné un idéal maximal $\mm$ dans $R$, on considère un corps valué par $\mm$, noté $\Kt_{\mm}$, compris entre $K$ et $\Kh_\mm$. Le corps $K$ est donc dense dans $\Kt_{\mm}$ pour la topologie $\mm$-adique. Ce corps est également supposé hensélien. On a donc : $K\subset K_{\mm}^h\subset \Kt_{\mm}\subset \Kh_{\mm}$. Son anneau d'entiers est noté $\Rt_\mm$.

Il arrive parfois dans l'article d'alléger les hypothèses faites sur les $\Kt_{\mm}$. Ceci est alors mentionné explicitement.
\medskip

L'ensemble $\Specm(R)$ désigne le spectre maximal de $R$, c'est-à-dire l'ensemble de ses idéaux maximaux. Observons que dans le cas de $R$, il s'agit également des idéaux premiers non nuls. Notons dans ce cas $\Rt:= \prod_{\mm\in \Specm(R)}\Rt_\mm$ et $\Kt:= \prod_{\mm\in \Specm(R)}\Kt_\mm$.
Pour tout $\mm\in \Specm(R)$, le corps résiduel de $\Rt_{\mm}$ est noté $\kappa_{\mm}$. Notons alors $\kappa:= \prod_{\mm\in \Specm(R)}\kappa_\mm$.
Observons que $R\subset \Rt$ et $K\subset \Kt$ au travers de l'inclusion diagonale. Cette inclusion est implicite tout le long du document.
\medskip

Dans le cas où l'on a $\Rt_{\mm} = \Rh_{\mm}$ pour tout $\mm\in \Specm(R)$, observons que $\Rh:=\Rt$ est aussi le complété de $R$ par son radical de Jacobson (cf. \cite[Theorem 8.15.]{Matsumura}).
\medskip

Notons également $\Ktnr:=\prod_{\mm\in \Specm(R)}\Ktnr_\mm$ et $\Rtnr:=\prod_{\mm\in \Specm(R)}\Rtnr_\mm$ respectivement les produits des extensions maximales non ramifiées et les produits des hensélisés stricts. Posons aussi $\kappa^s:=\prod_{\mm\in \Specm(R)}\kappa_\mm^s$, le produit des clôtures séparables. Notons également $\Galnr:= \prod_{\mm\in \Specm(R)}\Galnr_{\mm}:= \prod_{\mm\in \Specm(R)}\Gal(\Ktnr_\mm/\Kt_\mm)$.
\medskip

Soulignons que l'extension maximale non ramifiée d'un corps complet n'est pas toujours complète. Par exemple, l'extension maximale non ramifiée de $\kappa((t))$ n'est pas $\kappa^s((t))$ si $\kappa^s/\kappa$ est infini.
\medskip

Étant donné un groupe réductif $G$ sur $\Kt$ (ce qui est équivalent à se donner des groupes réductifs sur les corps $\Kt_{\mm}$), on note $\ImmBT(G):=\prod_{\mm\in \Specm(R)}\ImmBT(G_{\Kt_\mm})$, le produit des immeubles de Bruhat-Tits des $G_{\Kt_\mm}$ (ils existent d'après \cite[Proposition 1.1.]{Article1}). Le groupe $G(\Kt)=\prod_{\mm\in \Specm(R)}G(\Kt_{\mm})$ agit naturellement sur $\ImmBT(G):=\prod_{\mm\in \Specm(R)}\ImmBT(G_{\Kt_\mm})$. Une facette (resp. chambre, resp. appartement) dans $\ImmBT(G)$ est le produit de facettes (resp. chambres, resp. appartements) dans chacun des facteurs.

De la même manière, en considérant tout facteur par facteur, on généralise la notion de sous-groupes parahoriques, de sous-groupes stabilisateurs, de schémas en groupes de Bruhat-Tits, etc.
\medskip

Notons aussi que $\Galnr:=\prod_{\mm\in \Specm(R)}\Galnr_{\mm}$ agit naturellement sur $G(\Ktnr):=\prod_{\mm\in \Specm(R)}G(\Ktnr_{\mm})$. 

\pagebreak

\section{Découpage du problème et techniques de recollements}

On se donne comme objectif dans cette partie d'utiliser les techniques de recollements (ou \textit{patching} en anglais) pour séparer le problème qui nous intéresse en deux questions intermédiaires.
\medskip

Plus précisément, on reprend l'idée développée par Nisnevich (\cite{TheseNisnevich}, \cite{Nisnevich}) et Guo (\cite{Gu}). Autrement dit, essayer de se ramener au cas où $R$ est local et complet (ou toute autre situation plus élémentaire) et comprendre l'injectivité dans ce cas de figure. Ceci utilise donc les techniques de recollements.
\medskip

On a par ailleurs fait le choix dans cette section de travailler avec les espaces algébriques au lieu des schémas affines. En effet, comme les techniques de recollements ne sont pas disponibles pour les schémas quelconques, travailler avec les espaces algébriques permet de contourner cette difficulté et d'obtenir tout de même des résultats utiles pour notre problème. Dans une première approche, le lecteur peut donc considérer seulement des schémas affines.
\medskip

Dans cette partie, les corps valués $\Kt_{\mm}$ sont seulement supposés contenir $K$ et avoir les mêmes corps résiduels que $K$ sous les valuations $\mm$-adiques (c'est-à-dire les $\kappa_{\mm}$). Ils ne sont donc ni nécessairement henséliens, ni nécessairement dans $\Kh_\mm$.
\medskip

Soit $\Gc$ un schéma en groupes sur $R$ séparé et localement de présentation finie. Notons également $G:=\Gc_K$.

\begin{question} \label{MainDiagram}
Considérons le diagramme commutatif suivant :
\[\begin{tikzcd}
	{H^1_{*}(R,\Gc)} & {H^1_*(\Rt,\Gc)} & {\displaystyle\prod_{\mm\in \Specm(R)}{H^1_{*}(\Rt_{\mm},\Gc)}} \\
	{H^1_{*}(K,G)} & {H^1_*(\Kt,\Gc)} & {\displaystyle\prod_{\mm\in \Specm(R)}{H^1_{*}(\Kt_{\mm},G)}}
	\arrow[from=1-1, to=1-2]
	\arrow[from=1-1, to=2-1]
	\arrow[Rightarrow, no head, from=1-2, to=1-3]
	\arrow[from=1-2, to=2-2]
	\arrow[from=2-1, to=2-2]
	\arrow[Rightarrow, no head, from=2-2, to=2-3]
\end{tikzcd}\]
avec $*\in \{\mathrm{fppf}, \textrm{\upshape ét}\}$. Quelle est l'obstruction de ce diagramme à être cartésien ?
\end{question}

Faisons un rappel sur les techniques de recollements :

\begin{rap}[Techniques de recollements] \label{Patching}
    Le foncteur suivant est une équivalence de catégories : \upshape
    \begin{align*}
        \left\{
        \begin{gathered}
          \text{Catégorie des } R\text{-espaces }\\
          \text{algébriques séparés et}  \\
          \text{loc. de présentation finie}
        \end{gathered}
        \right\}
        & \rightarrow
        \left\{
        \begin{gathered}
        \text{Catégorie des triplets } (X',\Xf',\tau:X'_{\Kt}\rightarrow\Xf'_{\Kt})\\
        \text{où } X' \text{ (resp. } \Xf' \text{) est un espace algébrique séparé}\\
        \text{loc. de prés. finie sur }K \text{ (resp. } \Rt\text{) et }
         \tau \text{ un isomorphisme}
        \end{gathered}
        \right\}\\
        \Xf &\mapsto \left(\Xf_K,\Xf_{\Rt}, (\Xf_K)_{\Kt}\overset{\sim}{\rightarrow} (\Xf_{\Rt})_{\Kt}\right).
    \end{align*}
\end{rap}

\begin{proof}
    Il s'agit d'une conséquence de \cite[Corollaire 5.6.(1)]{MB}, puisque \linebreak $\Spec(\Rt)\rightarrow \Spec(R)$ est plat d'après \cite[Corollary 1.2.14.]{liu2006algebraic} et induit un isomorphisme au niveau des points fermés, puisque $R$ et $\Rt$ ont les mêmes corps résiduels.
\end{proof}

Par fonctorialité, on vérifie qu'un tel foncteur se restreint et corestreint aux espaces algébriques en groupes. D'après \cite[4.B. Théorème]{EspAlgGrp}, les espaces algébriques en groupes considérés sont représentables par des schémas. On obtient donc :

\begin{prop}
    Le foncteur suivant est une équivalence de catégories : \upshape
    \begin{align*}
        \left\{
        \begin{gathered}
          \text{Catégorie des } R\text{-schémas }\\
          \text{en groupes séparés et}  \\
          \text{loc. de présentation finie}
        \end{gathered}
        \right\}
        & \rightarrow
        \left\{
        \begin{gathered}
        \text{Catégorie des triplets } (G',\mathfrak{G}',\tau:G'_{\Kt}\rightarrow\mathfrak{G}'_{\Kt})\\
        \text{où } G' \text{ (resp. } \mathfrak{G}' \text{) est un schéma en groupes séparé}\\
        \text{loc. de prés. finie sur }K \text{ (resp. } \Rt\text{) et }
         \tau \text{ un isomorphisme}
        \end{gathered}
        \right\}\\
        \mathfrak{G}& \mapsto \left(\mathfrak{G}_K,\mathfrak{G}_{\Rt}, (\mathfrak{G}_K)_{\Kt}\overset{\sim}{\rightarrow} (\mathfrak{G}_{\Rt})_{\Kt}\right).
    \end{align*}
\end{prop}

Que dire maintenant des torseurs ? On a besoin de quelques lemmes qui sont d'ailleurs valides sur une base quelconque $S$ (qui est un schéma) ; $\Gc$ est donc supposé être un schéma en groupes sur $S$ non nécessairement séparé, ni nécessairement localement de présentation finie.

\begin{déf}
    On dit que $\Xf$, un faisceau fppf sur $S$, est un pseudo $\Gc$-torseur sur $S$ si $\Xf$ est muni d'une action libre et transitive de $\Gc$. Autrement dit, une action telle que $\Gc\times_S \Xf\rightarrow \Xf \times_S \Xf$, $(g,x)\mapsto (g.x,x)$ est un isomorphisme.
    
    On note $H^1_{\mathrm{Pseudo}}(S,\Gc)$ (resp. $H^1_{\mathrm{SLPF}}(S,\Gc)$) l'ensemble des classes d'isomorphismes de pseudo $\Gc$-torseurs sur $S$ représentables par des espaces algébriques (resp. de pseudo \linebreak $\Gc$-torseurs sur $S$ représentables par des espaces algébriques séparés et localement de présentation finie).

    Par ailleurs, on définit un $\Gc$-torseur sur $S$ pour la topologie fppf (resp. étale) comme étant un faisceau muni d'une action de $\Gc$, localement isomorphe à $\Gc$ muni de son action par translation (à gauche ou à droite selon la convention que l'on prend).

    On note $[\mathfrak{X}]$ la classe d'isomorphisme de $\mathfrak{X}$. Notons qu'il n'y a pas (ni ici, ni dans la suite) d’ambiguïté sur la catégorie ambiante dans notre situation. 
\end{déf}

\begin{rmq}
    Un pseudo $\Gc$-torseur sur $S$ est isomorphe au pseudo torseur trivial ($\Gc$ muni de son action par translation) si et seulement s'il admet une section sur $S$ (cf. \stacks{03AI}).
\end{rmq}

\begin{lem}\label{Torseurs}
    Tout torseur pour la topologie fppf/étale est représentable par un espace algébrique qui est un pseudo-torseur.
    On a donc les inclusions naturelles \linebreak $H^1_{\textrm{\upshape ét}}(S,\Gc)\subset H^1_\mathrm{fppf}(S,\Gc)\subset H^1_{\mathrm{Pseudo}}(S,\Gc)$. Plus précisément :    
    \begin{enumerate}
        \item Si $\Gc$ est plat et localement de présentation finie (resp. et aussi séparé), l'ensemble pointé $H^1_\mathrm{fppf}(S,\Gc)$ est égal à l'ensemble pointé des classes d'isomorphismes d'espaces algébriques pseudo $\Gc$-torseurs sur $S$ fidèlement plats et loc. de présentation finie (resp. et aussi séparé).
        
        \item Si $\Gc$ est lisse (resp. et aussi séparé), l'ensemble pointé $H^1_{\textrm{\upshape ét}}(S,\Gc)$ est égal à l'ensemble pointé des classes d'isomorphismes d'espaces algébriques pseudo $\Gc$-torseurs sur $S$ lisses et surjectifs (resp. et aussi séparé).
    \end{enumerate}
\end{lem}

\begin{proof}
    Notons que les torseurs pour la topologie fppf/étale sont représentables par des espaces algébriques car la descente fppf/étale est toujours effective pour eux (cf. \stacks{0ADV}).

    Montrer que $\Gc\times_S \Xf\rightarrow \Xf \times_S \Xf$ est un isomorphisme peut se faire après localisation fppf/étale. Comme les torseurs fppf/étales sont triviaux fppf/étale localement, on a donc le résultat.

    Comme être plat, localement de présentation finie, lisse ou encore séparé est local pour la topologie fppf ou étale, si $\Gc$ l'est, alors les torseurs fppf ou étales le sont. Notons d'ailleurs que $\Gc$ est toujours surjectif sur $S$ puisque le morphisme $\Gc\rightarrow S$ admet une section.
    \medskip

    Réciproquement, soit $\Xf$ un pseudo $\Gc$-torseur sur $S$. Notons que $\Xf\times_S \Xf\rightarrow \Xf$ est un $\Gc_{\Xf}$-torseur trivial puisqu'il possède une section.

    Considérons maintenant un morphisme étale surjectif $U\rightarrow \Xf$ où $U$ est représentable par un schéma. On en déduit que $(\Xf\times_S \Xf)\times_{\Xf} U=\Xf\times_S U \rightarrow U$ est également un $\Gc_U$-torseur trivial sur $U$.

    Donc $U\rightarrow S$ est un recouvrement trivialisant $\Xf$. Si $\Xf$ est fidèlement plat et localement de présentation finie, $U$ l'est également aussi par composition. Donc $\Xf$ est trivialisé par un recouvrement fppf. De même, si $\Xf$ est lisse et surjectif, $U$ aussi et donc $\Xf$ est trivialisé par un recouvrement étale.
\end{proof}

On en déduit donc le résultat suivant :

\begin{coro}
    Si $\Gc$ est lisse, alors $H^1_{\textrm{\upshape ét}}(S,\Gc)=H^1_\mathrm{fppf}(S,\Gc)$.
\end{coro}

\begin{proof}
    Un torseur fppf $\Xf$ est fppf localement isomorphe à $\Gc$. Par descente fppf, $\Xf$ est également lisse et surjectif. D'après le lemme précédent, $[\Xf]\in H^1_{\textrm{\upshape ét}}(S,\Gc)$.
\end{proof}

Revenons maintenant au cas où $S=\Spec(R)$ et $\Gc$ séparé et localement de présentation finie. On peut enfin énoncer les techniques de recollements pour les torseurs :

\begin{prop} \label{PatchingTorseurs}
    Le foncteur suivant est une équivalence de catégories : \upshape
    \begin{align*}
        \left\{
        \begin{gathered}
          \text{Catégorie des } R\text{-esp. alg. }\\
          \text{pseudo }\Gc\text{-torseur sép.}  \\
          \text{et loc. de prés. finie}
        \end{gathered}
        \right\}
        & \rightarrow
        \left\{
        \begin{gathered}
        \text{Catégorie des triplets } (X',\Xf',\tau:X'_{\Kt}\rightarrow\Xf'_{\Kt})\\
        \text{où } X' \text{ (resp. }\Xf'\text{) est un esp. alg. pseudo torseur sur }G\text{ (resp. }\Gc_{\Rt})\\
        \text{sép. loc. de prés. finie sur }K \text{ (resp. } \Rt\text{) et }
         \tau \text{ un isomorphisme}
        \end{gathered}
        \right\}\\
        \Xf &\mapsto \left(\Xf_K,\Xf_{\Rt}, (\Xf_K)_{\Kt}\overset{\sim}{\rightarrow} (\Xf_{\Rt})_{\Kt}\right).
    \end{align*}
    
    \textit{Si de plus $\Gc$ est plat (resp. lisse), alors l'équivalence de catégories précédente en induit également une au niveau des torseurs fppf (resp. torseurs étales).}
\end{prop}

\begin{proof}
    Le premier résultat est évident par définition des pseudo torseurs et par fonctorialité des techniques de recollements (rappel \ref{Patching}) : on peut restreindre et corestreindre sans difficulté.

    Pour le second résultat, lorsque $\Gc$ est plat (resp. lisse) on peut également restreindre et corestreindre aux pseudo torseurs qui sont de plus fidèlement plats (resp. lisses et surjectifs). En effet, comme $\Spec(\Rt)\rightarrow \Spec(R)$ est fidèlement plat et quasi-compact, si un pseudo torseur est tel que $\Xf_{\Rt}$ est fidèlement plat (resp. lisse et surjectif), alors $\Xf$ l'est également par descente fpqc.
    
    On a donc le résultat d'après le lemme \ref{Torseurs}.
\end{proof}

Utilisons donc les techniques de recollements pour reformuler notre problème. Nous donnons alors une variante de \cite[Théorème 2.1.]{Nisnevich}, ou encore de \cite[Proposition 10.]{Gu} :

\begin{thm} \label{DoubleQuotientPatching}
    Prenons $*\in \{\mathrm{SLPF}, \mathrm{fppf}, \textrm{\upshape ét}\}$ (supposant de plus que $\Gc$ est plat (resp. lisse) si $*=\mathrm{fppf}$ (resp. $\textrm{\upshape ét}$)). Désignons par $\tau_g : G_{\Kt}\cong G_{\Kt}$ l'isomorphisme de torseurs obtenu en translatant (à gauche) par un élément $g\in G(\Kt)$. 
    L'application $g\mapsto (G,\Gc_{\Rt},\tau_g)$ induit par recollement la bijection d'ensembles pointés suivant :
$${\Gc(\Rt)\backslash G(\Kt)/G(K)}\cong\Ker\left(H^1_*(R,\Gc)\rightarrow H^1_*(K,G)\times_{H^1_*(\Kt,G)}H^1_*(\Rt,\Gc)\right).$$
Par conséquent, on a la suite exacte naturelle :
\[\begin{tikzcd}
	1 & {\Gc(\Rt)\backslash G(\Kt)/G(K)} & 
{H^1_*(R,\Gc)} & 
{H^1_*(K,G)\times_{H^1_*(\Kt,G)}H^1_*(\Rt,\Gc)} & 1.
	\arrow[from=1-2, to=1-3]
	\arrow[from=1-1, to=1-2]
	\arrow[from=1-3, to=1-4]
	\arrow[from=1-4, to=1-5]
\end{tikzcd}\]
\end{thm}

\begin{proof}
    On a un morphisme naturel $H^1_*(R,\Gc)\rightarrow H^1_*(K,G)\times_{H^1_*(\Kt,G)}H^1_*(\Rt,\Gc)$ donné par $[\Xf]\rightarrow ([\Xf_{\Rt}],[\Xf_{\Kt}])$. Ce morphisme est en fait surjectif. En effet, prenons $([X],[\Xf'])$ dans le produit fibré. Par définition, $X$ et $\Xf'$ ont même classe dans $H^1_*(\Kt,G)$. Cela signifie qu'il existe un isomorphisme de torseurs $\tau : X_{\Kt}\rightarrow \Xf'_{\Kt}$. On peut donc utiliser les techniques de recollements (proposition \ref{PatchingTorseurs}) pour obtenir un $R$-torseur $\Xf$ sur $\Gc$ qui recolle $X$ et $\Xf'$ et donc tel que $[\Xf]$ s'envoie sur $([X],[\Xf'])$ comme souhaité. D'où la surjectivité.
    \medskip

    Que peut-on dire du noyau de ce morphisme ? On recherche donc les $R$-torseurs sur $\Gc$ qui sont triviaux sur $\Rt$ et $\Kt$, à isomorphisme près. D'après les techniques de recollements, cela revient à comprendre les triplets de la forme $(G,\Gc_{\Rt},\tau: G_{\Kt}\overset{\sim}{\rightarrow}(\Gc_{\Rt})_{\Kt}=G_{\Kt})$ à isomorphisme près. L'isomorphisme $\tau$ est d'ailleurs déterminé par l'image de l'élément neutre qui est un élément de $G(\Kt)$. Réciproquement, tout élément $g\in G(\Kt)$ détermine un isomorphisme $\tau_g$ en translatant par cet élément. Les triplets que l'on cherche sont donc exactement déterminés par un élément de $G(\Kt)$.
    \medskip

    Comprenons maintenant les triplets isomorphes. Un triplet $(G,\Gc_{\Rt},\tau_{\widehat{g}})$ est isomorphe à un triplet $(G,\Gc_{\Rt},\tau_{\widehat{g}'})$ si et seulement s'il existe $g\in G(K)$ et $p\in \Gc(\Rt)$ tel que le carré suivant commute :
    \[\begin{tikzcd}
    	{G_{\Kt}} & {({\Gc_{\Rt}})_{\Kt}} \\
    	{G_{\Kt}} & {({\Gc_{\Rt}})_{\Kt}}
    	\arrow["{\tau_{\widehat{g}}}", from=1-1, to=1-2]
    	\arrow["{\tau_g}"', from=1-1, to=2-1]
    	\arrow["{\tau_p}", from=1-2, to=2-2]
    	\arrow["{\tau_{\widehat{g'}}}"', from=2-1, to=2-2]
    \end{tikzcd}\]
    En d'autres termes, $\tau_{\widehat{g}'}=\tau_{p}\circ \tau_{\widehat{g}} \circ \tau_{g}^{-1}$.
    
    En évaluant en l'élément neutre, on a alors $\widehat{g}'=p\,\widehat{g}\,g^{-1}$, supposant que l'on manipule des torseurs à gauche. Les classes d'isomorphismes sont donc données par $\Gc(\Rt)\backslash G(\Kt)/G(K)$.
\end{proof}

\begin{rmq}
    Remarquons que ${\Gc(\Rt)\backslash G(\Kt)/G(K)}$ est en bijection d'ensembles pointés avec $G(K)\backslash G(\Kt)/{\Gc(\Rt)}$ grâce à $g\mapsto g^{-1}$. On obtient en fait l'un ou l'autre ensemble par les calculs précédents en fonction de si l'on souhaite travailler avec des torseurs à gauche ou à droite. Ce choix n'a aucune importance.
\end{rmq}

\begin{rmq}
    Il est intéressant de noter que, lorsque $\Gc$ est affine, $G$ est lisse, et $\Kt=\Kh$, le double quotient $\Gc(\Rh)\backslash G(\Kh)/G(K)$ est isomorphe à $H^1_{\mathrm{Nis}}(R,\Gc)$ d'après \cite[2.8. Theorem]{TheseNisnevich}. Si $\Gc$ est de plus plat, alors d'après \cite[1.3. Proposition]{TheseNisnevich}, on a \linebreak ${\Gc(R^h)\backslash G(K^h)/G(K)}={\Gc(\Rh)\backslash G(\Kh)/G(K)}$.
\end{rmq}

Notons que dire que le diagramme de la question \ref{MainDiagram} est cartésien est équivalent à dire que l'on a une bijection d'ensembles pointés $H^1(R,\Gc)\overset{\sim}{\rightarrow} H^1(\Rt,\Gc)\times_{H^1(\Kt,G)}H^1(K,G)$. \linebreak En particulier, $\Ker\left(H^1(R,\Gc)\rightarrow H^1(\Rt,\Gc)\times_{H^1(\Kt,G)}H^1(K,G)\right)$, donc ${\Gc(\Rt)\backslash G(\Kt)/G(K)}$, doit être trivial.
\medskip

Rappelons cependant que l'on manipule des ensembles pointés et non des groupes a priori. Par conséquent, le noyau ne suffit pas à comprendre les fibres du morphisme. 

Toutefois, les techniques dites de torsion (ou \textit{twist} en anglais) permettent de comprendre ses fibres. Supposons désormais à partir de maintenant que $\Gc$ est plat. Il peut alors être identifié avec le faisceau fppf qu'il représente. Par ailleurs, dans toute la suite, on prend $*\in \{\mathrm{fppf}, \textrm{\upshape ét}\}$ (en supposant que $\Gc$ est de plus lisse si $*=\textrm{\upshape ét}$). Faisons quelques rappels :
\medskip

Prenons un torseur $\Xf$ avec donc $[\Xf] \in H^1_*(R,\Gc)$, et considérons le groupe tordu de $\Gc$ par $\Xf$ par automorphismes intérieurs, noté $\Gc^\Xf$ (cf. \cite[2.1.]{GilleSemisimpleSchémas}). C'est une forme fppf (ou étale) sur $R$ de $\Gc$ telle que sa classe dans $H^1_*(R,\Aut(\Gc)))$ est donnée par l'image de $[\Xf]$ par l'application naturelle $H^1_*(R,\Gc)\rightarrow H^1_*(R,\Aut(\Gc)))$ (cf. \cite[Chapitre III, Corollaire 2.5.4.]{Giraud}). En conséquence, deux torseurs isomorphes induisent des tordus isomorphes.

Il est aussi tel qu'il existe une bijection canonique $\varphi_\Xf$ de $H^1_*(R,\Gc^\Xf)$ dans $H^1_*(R,\Gc)$ qui envoie la classe du torseur trivial vers $[\Xf]$ (cf. \cite[Chapitre III, 2.6.]{Giraud}). Notons par ailleurs que tordre $\Gc^\Xf$ par un torseur $\mathfrak{Y}$ avec donc $[\mathfrak{Y}] \in H^1_*(R,\Gc^\Xf)$ donne à isomorphisme près le même groupe que si l'on tordait $\Gc$ par un torseur dans la classe $\varphi_\Xf([\mathfrak{Y}]) \in H^1_*(R,\Gc)$.

Enfin, observons que grâce à \cite[4.A. Théorème]{EspAlgGrp}, un tordu fppf/étale d'un schéma en groupes sur $R$, plat, séparé et loc. de présentation finie (qui, par descente, est un $R$-espace algébrique en groupes plat, séparé et loc. de présentation finie) est en fait représentable par un $R$-schéma en groupes. Dans la suite, on peut alors réutiliser pour les tordus de $\Gc$ ce que l'on a déjà fait.
\medskip 

De ceci, on en déduit les lemmes suivants :

\begin{lem}\label{TwistEtFibresKer}
    Soit $[\Xf] \in H^1_*(R,\Gc)$. On a : $$\Ker\left(H^1_*(R,\Gc^\Xf)\rightarrow H^1_*(K,\Gc^\Xf)\right)\cong g^{-1}(g([\Xf])),$$ où $g$ désigne $H^1_*(R,\Gc)\rightarrow H^1_*(K,G)$.
\end{lem}

\begin{proof}
    On a le diagramme commutatif suivant à flèches verticales bijectives :
    \[\begin{adjustbox}{max size={1\textwidth}{1\textheight}}
    \begin{tikzcd}
    	1 & {\Ker\left(H^1_*(R,\Gc^\Xf)\rightarrow H^1_*(K,\Gc^\Xf)\right)} & {H^1_*(R,\Gc^\Xf)} & {H^1_*(K,\Gc^\Xf)} & 1 \\
    	1 & {g^{-1}(g([\Xf]))} & {H^1_*(R,\Gc)} & {H^1_*(K,G)} & 1.
    	\arrow[from=1-1, to=1-2]
        \arrow[from=1-2, to=2-2]
    	\arrow[from=1-2, to=1-3]
    	\arrow[from=1-3, to=1-4]
    	\arrow["{1\mapsto [\Xf]}", from=1-3, to=2-3]
    	\arrow[from=1-4, to=1-5]
    	\arrow["{1\mapsto [\Xf_K]}", from=1-4, to=2-4]
    	\arrow[from=2-1, to=2-2]
    	\arrow[from=2-2, to=2-3]
    	\arrow["f",from=2-3, to=2-4]
    	\arrow[from=2-4, to=2-5]
    \end{tikzcd}
    \end{adjustbox}\]
   La première ligne est exacte. La seconde ligne l'est également en choisissant $[\Xf]$ et $[\Xf_K]$ comme éléments neutres. D'où le résultat.    
\end{proof}

\begin{lem}\label{TwistEtFibres}
    Soit $[\Xf] \in H^1_*(R,\Gc)$. On a : $${(\Gc^\Xf)(\Rt)\backslash (\Gc^\Xf)(\Kt)/(\Gc^\Xf)(K)}\cong f^{-1}(f([\Xf])),$$ où $f$ désigne $H^1_*(R,\Gc)\rightarrow H^1_*(\Rt,\Gc)\times_{H^1_*(\Kt,G)}H^1_*(K,G)$.
\end{lem}

\begin{proof}
    On a le diagramme commutatif suivant à flèches verticales bijectives :
    \[\begin{adjustbox}{max size={1\textwidth}{1\textheight}}
    \begin{tikzcd}
    	1 & {(\Gc^\Xf)(\Rt)\backslash (\Gc^\Xf)(\Kt)/(\Gc^\Xf)(K)} & {H^1_*(R,\Gc^\Xf)} & {H^1_*(\Rt,\Gc^\Xf)\times_{H^1_*(\Kt,\Gc^\Xf)}H^1_*(K,\Gc^\Xf)} & 1 \\
    	1 & {f^{-1}(f([\Xf]))} & {H^1_*(R,\Gc)} & {H^1_*(\Rt,\Gc)\times_{H^1_*(\Kt,G)}H^1_*(K,G)} & 1.
    	\arrow[from=1-1, to=1-2]
        \arrow[from=1-2, to=2-2]
    	\arrow[from=1-2, to=1-3]
    	\arrow[from=1-3, to=1-4]
    	\arrow["{1\mapsto [\Xf]}", from=1-3, to=2-3]
    	\arrow[from=1-4, to=1-5]
    	\arrow["{(1,1)\mapsto([\Xf_{\Rt}],[\Xf_K])}", from=1-4, to=2-4]
    	\arrow[from=2-1, to=2-2]
    	\arrow[from=2-2, to=2-3]
    	\arrow["f",from=2-3, to=2-4]
    	\arrow[from=2-4, to=2-5]
    \end{tikzcd}
    \end{adjustbox}\]
   La première ligne est exacte. La seconde ligne l'est également en choisissant $[\Xf]$ et $([\Xf_{\Rt}],[\Xf_K])$ comme éléments neutres. D'où le résultat.
\end{proof}

Cela nous permet notamment d'avoir des résultats sur les noyaux des flèches du diagramme de la question \ref{MainDiagram} :

\pagebreak

\begin{prop}\label{CritereEgaliteKer}
    Notons $\mathfrak{C}$, l'ensemble des $\Gc^\Xf$ pour $[\Xf]$ parcourant l'ensemble \linebreak $\Ker\left(H^1_*(R,\Gc)\rightarrow H^1_*(K,G)\right)$ (en choisissant qu'un seul représentant pour chaque classe d'isomorphisme). On a :
    \begin{equation*}
    \begin{array}{c}
    \forall \Gc'\in \mathfrak{C},\\ 
    {\Gc'(\Rt)\backslash \Gc'(\Kt)/\Gc'(K)}\\
    \text{\upshape est trivial.}
    \end{array}  
    \iff
    \begin{array}{c}
    \text{\upshape les noyaux de } H^1_*(R,\Gc)\rightarrow H^1_*(K,G)\\
    \text{\upshape et } H^1_*(\Rt,\Gc)\rightarrow H^1_*(\Kt,G)\\
    \text{\upshape sont en bijection naturelle.}
    \end{array}   
\end{equation*}
\end{prop}

\begin{proof}
    Soit $[\Xf']\in \Ker\left(H^1_*(\Rt,\Gc)\rightarrow H^1_*(\Kt,\Gc)\right)$. D'après le théorème précédent, \linebreak le couple $(1,[\mathfrak{X'}])$ dans $H^1_*(K,G)\times_{H^1_*(\Kt,G)}H^1_*(\Rt,\Gc)$ provient d'une classe $[\Xf] \in H^1_*(\Rt,\Gc)$. Par définition, son image dans $H^1_*(K,G)$ est triviale. D'où la surjectivité.

    Soit un élément $[\Xf]\in \Ker\left(H^1_*(R,\Gc)\rightarrow H^1_*(K,\Gc)\right)$. D'après le lemme \ref{TwistEtFibres}, on a l'isomorphisme $(\Gc^\Xf)(\Rt)\backslash (\Gc^\Xf)(\Kt)/(\Gc^\Xf)(K)\cong f^{-1}(f([\Xf]))$. En conséquence, le double quotient est trivial si et seulement si $f^{-1}(f([\Xf]))$ est trivial ; c'est-à-dire si et seulement si $[\Xf]$ est l'unique élément de $H^1_*(R,\Gc)$ qui s'envoie sur $([\Xf_{\Rt}],1)$ par $f$, ou encore, si et seulement si c'est l'unique élément de $\Ker\left(H^1_*(R,\Gc)\rightarrow H^1_*(K,\Gc)\right)$ valant $[\Xf_{\Rt}]$ dans $H^1_*(\Rt,\Gc)$. Ceci prouve l'équivalence.
\end{proof}

On en déduit finalement :

\begin{thm} \label{DecoupagePb}
    Notons $\mathfrak{C}$, l'ensemble des $\Gc^\Xf$ pour $[\Xf]$ parcourant l'ensemble $H^1_*(R,\Gc)$ (en choisissant qu'un seul représentant pour chaque classe d'isomorphisme). On a :
\begin{equation*}
    \begin{array}{c}
    \forall \Gc'\in \mathfrak{C},\\ 
    {\Gc'(\Rt)\backslash \Gc'(\Kt)/\Gc'(K)}\\
    \text{\upshape est trivial.}
    \end{array}  
    \Leftrightarrow
    \begin{array}{c}
    \forall \Gc'\in \mathfrak{C}, \text{\upshape les noyaux de }\\
    H^1_*(R,\Gc')\rightarrow H^1_*(K,\Gc')\\
    \text{\upshape et } H^1_*(\Rt,\Gc')\rightarrow H^1_*(\Kt,\Gc')\\
    \text{\upshape sont en bijection naturelle.}
    \end{array}
    \Leftrightarrow
    \begin{array}{c}
    \text{\upshape les fibres de }\\
    H^1_*(R,\Gc)\rightarrow H^1_*(K,G)\\
    \text{\upshape et } H^1_*(\Rt,\Gc)\rightarrow H^1_*(\Kt,G)\\
    \text{\upshape sont en bijection naturelle.}
    \end{array}
    \Leftrightarrow
    \begin{array}{c}
    \text{\upshape le diagramme de}\\
    \text{\upshape la question \ref{MainDiagram}}\\ 
    \text{\upshape est cartésien.}
    \end{array}  
\end{equation*}
\end{thm}

\begin{proof}
    La première équivalence est une conséquence immédiate de la proposition \ref{CritereEgaliteKer}. La deuxième équivalence provient du lemme \ref{TwistEtFibresKer}. Enfin, la dernière équivalence est un résultat classique sur les diagrammes cartésiens.
\end{proof}

En résumé, pour répondre positivement à la question d'injectivité \ref{QuestionBFF} pour une certaine classe de groupes de Bruhat-Tits, grâce au théorème \ref{DecoupagePb}, la stratégie est d'établir les trois faits suivants :
\begin{enumerate}
    \item La classe des groupes de Bruhat-Tits que l'on étudie est stable par torsion intérieure\\
    (ces groupes sont toujours lisses et séparés) ;
    \item Le double quotient est trivial pour tout élément de cette classe ;
    \item La trivialité du noyau est réalisée sur $\Rt$ pour tout élément de cette classe.
\end{enumerate}
Bien entendu, on peut envisager une stratégie analogue si on est seulement intéressé par la trivialité du noyau grâce à la proposition \ref{CritereEgaliteKer}.

\begin{rmq}
    Comme annoncé en début de section, le lecteur peut éviter la notion d'espace algébrique en se limitant aux schémas affines (par exemple si le groupe étudié est semi-simple). Les preuves peuvent alors être simplifiées. En effet, on utilise d'une part que toute descente fpqc est effective pour les schémas affines, et d'autre part les techniques de recollements au niveau des schémas affines (cf. \cite[Theorème 1.1]{MB}).
\end{rmq}

\pagebreak

\begin{rmq}\label{RmqIndQuasiAff}
    Le point de vue des schémas ind-quasi-affines (\stacks{0AP5}) ne couvre pas non plus tous les cas qui nous intéressent bien qu'ils vérifient également la descente fpqc (\stacks{0APK}) et les techniques de recollements (cf. plus bas). En effet, le modèle de Néron $\Gc_m$ du tore $\GG_m$ (exemple simple d'un schéma en groupes de Bruhat-Tits non affine) n'est pas ind-quasi-affine comme nous allons l'établir ci-dessous (preuve communiquée par Gabber).
    \medskip
    
    Prenons $R$ local d'uniformisante $\pi$ pour simplifier. Il suffit de voir que l'union, que l'on note $U$, de $\pi^a \GG_{m,R}$, $\pi^b \GG_{m,R}$ et $\pi^c \GG_{m,R}$ dans $\Gc_m$ pour un choix $a,b,c$ d'entiers tous différents, n'est pas quasi-affine. En effet, $U$ est quasi-compact, donc le caractère ind-quasi-affine devrait impliquer que $U$ est quasi-affine par définition. Cela signifierait que $U\rightarrow \Spec(\Oc_U(R))$ est une immersion ouverte (cf. (4) de \stacks{01SM}) et donc que \linebreak $\pi^b \GG_{m,R} \rightarrow U\rightarrow \Spec(\Oc_U(R))$ l'est également.
    
    Par exemple, dans le cas où $(a,b,c)=(0,1,2)$, l'anneau des fonctions globales de $U$ vaut $R[X,\pi^2X^{-1}]$. En effet, le corps des fonctions de $\Gc_m$ est exactement $K(X,X^{-1})$. Les fonctions définies sur $\GG_{m,R}$ sont alors $R[X,X^{-1}]$. Pour être défini également sur $\pi \GG_{m,R}$ et $\pi^2\GG_{m,R}$, il faut préserver $\pi R^\times$ et $\pi^2 R^\times$. On réalise alors que les fonctions en question sont exactement $R[X,\pi^2 X^{-1}]$.

    Pour ce qui est de $\pi \GG_{m,R}$, on réalise qu'il s'agit de $R[\pi^{-1}X,\pi X^{-1}]$. Le morphisme $\pi \GG_{m,R}\rightarrow \Spec(\Oc_U(R))$ est alors donné au niveau des algèbres par l'inclusion \linebreak $R[X,\pi^2/X]\subset R[\pi^{-1}X,\pi X^{-1}]$.
    
    De manière un peu plus formelle, cela donne le morphisme suivant :
    \begin{align*}
        R[Y_1,Y_2]/(Y_1Y_2-\pi^2) & \overset{\varphi}{\rightarrow} R[Z_1,Z_2]/(Z_1Z_2-1)\\
        Y_1,Y_2&\mapsto \pi Z_1,\pi Z_2
    \end{align*}
    Au niveau des fibres spéciales, on a donc :
    \begin{align*}
        \kappa[Y_1,Y_2]/(Y_1Y_2) & \overset{\varphi}{\rightarrow} \kappa[Z_1,Z_2]/(Z_1Z_2-1)\\
        Y_1,Y_2&\mapsto 0,0
    \end{align*}
    Autrement dit, on a la factorisation : $(\pi\GG_{m,R})_{\kappa}\rightarrow \Spec(\kappa)\rightarrow \Spec(\Oc_U(R))_\kappa$. 
    
    Le morphisme $\pi \GG_{m,R}\rightarrow \Spec(\Oc_U(R))$ ne peut donc pas être une immersion ouverte, puisque cela n'est pas le cas sur $\kappa$.
\end{rmq}

\begin{prop}[Techniques de recollements sur les schémas ind-quasi-affines] ${}$ \linebreak
        Notons $\mathrm{INDQAFF}$, la catégorie fibrée des espaces algébriques ind-quasi-affines et reprenons le contexte de \cite[0.9]{MB} et l'hypothèse de platitude en \cite[1.0]{MB}. Le foncteur $\Phi_{\mathrm{INDQAFF}/S}$ est une équivalence de catégories.
\end{prop}

\begin{proof}
        C'est une conséquence immédiate de \cite[Corollaire 5.4.4.]{MB} puisque les ind-quasi-affines sont séparés, vérifient la descente fpqc (\stacks{0APK}) et que c'est une propriété locale sur la base pour la topologie fpqc (\stacks{0AP8}).
\end{proof}

\section{Techniques d'approximation}

Dans cette partie, $K$ désigne un corps infini (non nécessairement le corps de fractions d'un anneau de Dedekind semi-local). Soit $G$ un groupe algébrique réductif sur $K$.

On considère $\Sigma$ un ensemble non vide (éventuellement infini) de valuations discrètes non triviales de $K$ deux à deux non équivalentes. 
Posons $K_\Sigma := \prod_{v\in \Sigma}K_v$, où les $K_v$ sont des corps henséliens pour la valuation $v$ contenant $K$. On suppose par ailleurs que $K$ est dense dans chacun des $K_v$. Posons alors $G(K_\Sigma):=\prod_{v\in \Sigma}G(K_v)$. Pour tout $v\in \Sigma$, on voit également $G(K_v)$ dans $G(K_\Sigma)$ en l'identifiant avec $G(K_v)\times \prod_{w\in \Sigma\backslash\{v\}}\{1\}$.

Notons d'ailleurs que les $G(K_v)$ sont munis de la topologie adique (cf. \cite[3.1]{GGMB}).

Rappelons que la notation $G(K)^+$ désigne le sous-groupe de $G(K)$ engendré par les $K$-points des groupes de racines de $G$ (réduit à $\{1\}$ s'il y en a pas), ou encore par les $K$-points des sous-groupes unipotents déployés de $G$, et que $RG(K)$ désigne l'ensemble des éléments $R$-équivalents à l'élément neutre dans $G(K)$ (cf. \cite[\S 3]{CTS}). Notons alors $G(K_\Sigma)^+:=\prod_{v\in \Sigma}G(K_v)^+$ et $RG(K_\Sigma):=\prod_{v\in \Sigma}RG(K_v)$.
\medskip

L'objectif de cette partie est de montrer que $G(K_\Sigma)^+\subset \overline{G(K)}$. La motivation sous-jacente étant que $G(K_\Sigma)^+$ est un objet à la fois très maniable et suffisamment gros dans $G(K_\Sigma)$ pour nous aider à montrer la trivialité du double quotient de la partie précédente.
On a même mieux. Désignons par $\overline{RG(K)}$ l'adhérence de $RG(K)$ dans $G(K_\Sigma)$. On va montrer que $G(K_\Sigma)^+ \subset \overline{RG(K)}$.
\medskip

Pour tout $v\in \Sigma$, considérons donc les sous-groupes $K_v$-presque simples $G_{v,i}$ de $D(G)_{K_v}$, pour $i$ dans un ensemble fini $I_v$ (cf. \cite[Theorem 21.51.]{Milne}).
\medskip

La proposition suivante de Prasad va jouer un rôle crucial :

\begin{prop}[{\cite[Proposition 2.2.14]{kaletha_prasad_2023}}] \label{Prasad}
    Soit $L$ un corps valué discrètement hensélien et $H$ un $L$-groupe $L$-presque simple. Tout sous-groupe ouvert non borné de $H(L)$ contient le sous-groupe $H(L)^+$.
\end{prop}

On va donc montrer que l'on est bien dans le cadre de validité de cette proposition.
Pour cela, on a besoin de montrer quelques lemmes.
\medskip

Commençons par le lemme suivant bien connu dont on rappelle la preuve.

\begin{lem}\label{DecompLie}
     Soit $H$ un groupe réductif sur un corps infini $L$ et $T$ un tore maximal de $H$. Il existe $h_1,..,h_n \in H(L)$ tel que $\mathrm{Lie}(H)=\sum_{i=1}^n {}^{h_i}\mathrm{Lie}(T)$.
\end{lem}

\begin{proof}
    Dans la suite, on utilise le gras pour désigner le schéma vectoriel sous-jacent à un espace vectoriel. Considérons l'application :
    \begin{align*}
        H\times \mathbf{Lie}(T)& \overset{\varphi}{\rightarrow} \mathbf{Lie}(H)\\
        (h,t)&\mapsto \mathrm{ad}(h)(t)
    \end{align*}

    Elle est dominante d'après l'implication $(i)\implies (iv)$ de \sga{Exp. XIII, Théorème 5.1.} puisque les sous-groupes de Cartan dans un groupe réductif sont exactement les tores maximaux.
    
    En conséquence, la $H$-enveloppe de $\mathbf{Lie}(T)$, - c'est à dire le plus petit sous-schéma vectoriel de $\mathbf{Lie}(H)$ contenant $\mathbf{Lie}(T)$ sur lequel $H$ agit -, est exactement $\mathbf{Lie}(H)$.

    Notons $E:=\sum_{h\in H(L)}{}^h \mathrm{Lie}(T)$. C'est la $H(L)$-enveloppe de $\mathrm{Lie}(T)$. Montrons alors que $\mathbf{E}$ est $H$-stable. Par définition, $\mathbf{E}$ est $H(L)$-stable. Comme la $H$-stabilité est une condition fermée et que $H(L)$ est dense dans $H$ (puisque $H$ est unirationnel), on a la $H$-stabilité de $\mathbf{E}$ comme voulu.
    
    Par conséquent, $\mathbf{E}=\mathbf{Lie}(H)$, et donc $E=\mathrm{Lie}(H)$. Comme $\mathrm{Lie}(H)$ est de dimension finie, la somme définissant $E_0$ contient un nombre fini de termes. Ceci prouve le résultat.
\end{proof}

Montrons maintenant que $G_{v,i}(K_v)\cap \overline{RG(K)}$ est ouvert pour tous les $G_{v,i}$.

\begin{lem} \label{RGOuvert}
    Soit $v\in \Sigma$. Le sous-groupe $\overline{RG(K)}\cap G(K_v)$ est ouvert dans $G(K_v)$\\ (et donc $\overline{RG(K)}$ est ouvert dans $G(K_\Sigma)$ quand $\Sigma$ est fini).\\
    En particulier, pour tout $i\in I_v$, $G_{v,i}(K_v)\cap \overline{RG(K)}$ est un sous-groupe ouvert de $G_{v,i}(K_v)$.
\end{lem}

\pagebreak

\begin{proof}
    On utilise la technique de Raghunathan (qui provient de \cite[1.2]{RagTrick}). Considérons un $K$-tore maximal $T$ de $G$. D'après \ref{DecompLie}, il existe $g_1,..g_n$ tel que \linebreak $\mathrm{Lie}(G)=\sum_{i=1}^n {}^{g_i}\mathrm{Lie}(T)$. Prenons une suite exacte de tores $1\rightarrow S \rightarrow E \overset{\pi}{\rightarrow} T\rightarrow 1$ où $E$ est quasi-trivial (par ex. une résolution flasque de $T$).
    \medskip

    On peut donc considérer le morphisme (seulement de schémas !) :  
    $$\begin{aligned}
         f: E^n & \longrightarrow G \\
        (x_i) & \longmapsto {}^{g_1}\pi(x_1)\cdot ... \cdot {}^{g_n}\pi(x_n)
      \end{aligned}$$
    
    On a alors le diagramme commutatif suivant :
    \[\begin{tikzcd}
    	{\mathrm{Lie}(E^n)} & {\mathrm{Lie}(G)} \\
    	& {\mathrm{Lie}(T^n)}
    	\arrow["{\mathrm{Lie}(f)}", from=1-1, to=1-2]
    	\arrow["{\mathrm{Lie}(\pi^n)}"', two heads, from=1-1, to=2-2]
    	\arrow["{(x_i)\mapsto \sum_{i=1}^n{}^{g_i}x_i}"', two heads, from=2-2, to=1-2]
    \end{tikzcd}\]
    où l'on sait d'une part que $\mathrm{Lie}(E)\rightarrow \mathrm{Lie}(T)$ est surjectif puisque $\pi$ est lisse car $S$ l'est ; et d'autre part $\mathrm{Lie}(T^n)\rightarrow \mathrm{Lie}(G)$ est surjectif puisque $\mathrm{Lie}(G)=\bigoplus_{i=1}^n {}^{g_i}\mathrm{Lie}(T)$. On en déduit alors que $\mathrm{Lie}(f)$ l'est. Cela montre que $f$ est lisse au voisinage de l'élément neutre.
    
    D'après \cite[3.1.2 Lemme]{GGMB}, pour tout $v\in \Sigma$, il existe un ouvert $\Omega_v\subset G(K_v)$ tel que $f^{-1}(\Omega_v)\rightarrow \Omega_v$ admette une section. Donc $\Omega_v\subset f(E(K_v)^n)$.
    \medskip
    
    Comme $E$ est quasi-trivial, il est $K$-rationnel (cf. \cite[Proposition 12.64.]{Milne}). Par conséquent, d'après \cite[Proposition 2.1.]{RmqApproxFaibleCT}, $E(K)$ est dense dans $\prod_{v\in \Sigma}E(K_v)$. On en déduit que $f(E(K)^n)$ est également dense dans $\prod_{v\in \Sigma}f(E(K_v)^n)$. Donc comme \linebreak $f(E(K)^n)\subset RG(K)$, $\overline{RG(K)}$ contient $\prod_{v\in \Sigma}f(E(K_v)^n)$ et en particulier $\prod_{v\in \Sigma}\Omega_v$. 
    \medskip
    
    Comme $\overline{RG(K)}\cap G(K_v)$ contient l'ouvert non vide $\Omega_v$, c'est un sous-groupe ouvert de $G(K_v)$.
\end{proof}

On se propose ensuite de montrer que $G_{v,i}(K_v)\cap \overline{RG(K)}$ est non borné pour un éventuel $G_{v,i}$ isotrope sur $K_v$.
Pour cela, on va s'aider d'un lemme sur les tores :

\begin{lem}\label{ToreDensité}
    Soit $T$ un $K$-tore. On a $RT(K_\Sigma)\subset \overline{RT(K)}$.
\end{lem}

\begin{proof}
    Prenons une résolution flasque $1\rightarrow S\rightarrow E \rightarrow T\rightarrow 1$ de $T$. On sait que $E(K)$ est dense dans $E(K_\Sigma)$ par quasi-trivialité. Comme l'image de $E(K)$ dans $T(K)$ (resp. de $E(K_\Sigma)$ dans $T(K_\Sigma)$) est $RT(K)$ (resp. $RT(K_\Sigma)$), on a $RT(K_\Sigma)\subset \overline{RT(K)}$ (car $E(K_\Sigma)\rightarrow T(K_\Sigma)$ est continu pour la topologie adique d'après \cite[3.1.(ii)]{GGMB}).
\end{proof}

\begin{coro} \label{RGNonBorné}
    $G_{v,i}$ est $K_v$-isotrope si et seulement si $\overline{RG(K)}\cap G_{v,i}(K_v)$ est non borné.
\end{coro}

\begin{proof}
    Le sens réciproque est évident d'après \cite[Theorem 2.2.9]{kaletha_prasad_2023}. Regardons le sens direct.

    Prenons $T_i\subset G_{v,i}$ un tore isotrope de $G_{v,i}$. Il est inclus dans un tore maximal $T$ de $D(G)_{K_v}$. Comme $D(G)$ est défini sur $K$, par approximation faible des tores (cf. par ex. la preuve de \cite[Lemma 2.]{Gu}), il existe $g\in D(G)(K_v)$ tel que $T':=g T g^{-1}$ soit défini sur $K$. Notons $T'_i := g T_i g^{-1}$. Comme $G_{v,i}$ est distingué dans $D(G)_{K_v}$, $T'_i$ est un tore de $G_{v,i}$ qui est d'ailleurs isotrope puisque $T_i$ l'est. Prenons un $\GG_m$ inclus dans $T'_i$. Comme $\GG_m$ est déployé, il est $R$-trivial. Par conséquent, on a la suite d'inclusions d'après le lemme \ref{ToreDensité} :
    $$K_v^\times = \GG_m(K_v) \subset RT_i'(K_v)\subset RT_i'(K_\Sigma) \subset RT'(K_\Sigma) \subset \overline{RT'(K)}.$$

    Comme $T'(K_\Sigma)$ est fermé dans $G(K_\Sigma)$, la notation $\overline{RT'(K)}$ désigne le même objet, que l'on se place dans $T'(K_\Sigma)$ ou bien dans $G(K_\Sigma)$. Comme on a évidemment $RT'(K)\subset RG(K)$, on a $\overline{RT'(K)}\subset \overline{RG(K)}$. Mais donc, $K_v^\times$ appartient à $\overline{RG(K)}\cap G_{v,i}(K_v)$, on en déduit que ce dernier est non borné comme voulu !
\end{proof}

On a donc enfin prouvé les lemmes nécessaires à notre théorème :

\begin{thm}\label{AdherenceRG}
    Soit $G$ un $K$-groupe réductif. On a :
    $$G(K_\Sigma)^+ = D(G)(K_\Sigma)^+ \subset \overline{RD(G)(K)}\subset \overline{RG(K)}\subset \overline{G(K)}\subset G(K_\Sigma).$$
\end{thm}

\begin{proof}
    Prenons $v\in \Sigma$. On a que $G_{v,i}(K_v)\cap \overline{RG(K)}$ est sous-groupe ouvert (d'après le lemme \ref{RGOuvert}) non borné (d'après le corollaire \ref{RGNonBorné}) de $G_{v,i}(K_v)$ pour tout $G_{v,i}$ isotrope. Ceci implique alors que $G_{v,i}(K_v)^+\subset \overline{RG(K)}$ d'après la proposition \ref{Prasad}. 
    
    Observons ensuite d'après \cite[Theorem 21.51.]{Milne}, le morphisme naturel \linebreak $\prod_{i\in I_v} G_{v,i}\rightarrow D(G)_{K_v}$ est une isogénie. D'après \cite[Corollaire 6.3.]{HomAbstraits}, il envoie surjectivement $\prod_{i\in I_v} G_{v,i}(K_v)^+$ sur $D(G)(K_v)^+$. Autrement dit, les $G_{v,i}(K_v)^+$ engendrent $D(G)(K_v)^+$. Par ailleurs, notons que $D(G)(K_\Sigma)^+$ et $G(K_\Sigma)^+$ sont les mêmes groupes dans $G(K_\Sigma)$ grâce à \cite[Corollaire 6.3.]{HomAbstraits} appliqué à $D(G)\rightarrow G$. 
    
    On conclut donc des deux paragraphes précédents que $G(K_\Sigma)^+\subset \overline{RG(K)}$. En appliquant ce que l'on vient de faire pour $G=D(G)$, on trouve $D(G)(K_\Sigma)^+ \subset \overline{RD(G)(K)}$. Il suffit alors d'utiliser que $RD(G)(K)\subset RG(K)\subset G(K)$ et que les inclusions passent à l'adhérence pour en déduire le théorème.
\end{proof}

\begin{rmq}
    De toute évidence, $RG(K)$ et donc $\overline{RG(K)}$ est inclus dans $\prod_{v\in \Sigma}RG(K_v)$ (un produit quelconque de fermés est fermé). Par conséquent, si $G$ est semi-simple simplement connexe, le théorème précédent dit que, si pour tout $v\in \Sigma$, $G_{K_v}$ est strictement isotrope, alors $\prod_{v\in \Sigma}RG(K_v)=G(K_\Sigma)^+ = \overline{RG(K)}$ (d'après \cite[Théorème 7.2.]{BourbakiGille}). Dans le cas où on a un $G_{v,i}$ anisotrope, on ne sait pas si $RG_{v,i}(K_v)\subset \overline{RG(K)}$ ; cela impliquerait l'égalité $\prod_{v\in \Sigma}RG(K_v)=\overline{RG(K)}$ en toute généralité (puisque dans ce cas, $G_{K_v}=\prod_{i\in I_v} G_{v,i}$, cf. \cite[Theorem 24.3.]{Milne}). 
    \medskip

    Il y a toutefois un cas où l'on peut conclure que l'égalité est effectivement réalisée :
    \medskip
    
    Pour un groupe de type $^{1}A_n$, c'est-à-dire de la forme $G:=\mathrm{SL}_1(D)$, où $D$ est une algèbre à division de dimension finie sur $K$, on a bien $RG(K_v)\subset \overline{RG(K)}$. En effet, \linebreak $RG(K)=[D^\times,D^\times]$ d'après \cite{Voskresenskii}. De même, en posant $D_v:=D\otimes_KK_v$, \linebreak $RG(K_v)=[D_v^\times,D_v^\times]$. Par conséquent, le fait que $[D_v^\times,D_v^\times]\subset \overline{[D^\times,D^\times]}$ (puisque $D^\times$ vérifie l'approximation faible) donne le résultat.
\end{rmq}

On a également la proposition suivante en complément :

\begin{prop}\label{ToresRG}
    Soit $T$ un $K$-tore de $G$. On a $RT(K_\Sigma)\subset \overline{RG(K)}$. En particulier, si $T$ est $R$-trivial (par ex. si $T$ est déployé), alors on a $T(K_\Sigma)\subset \overline{RG(K)}$.

    Par ailleurs, pour $T$ un $K_\Sigma$-tore inclus dans $G_{K_\Sigma}$ (c'est à dire la donnée de tores dans chaque $G_{K_v}$), il existe $g\in G(K_\Sigma)$ tel que l'on ait $gRT(K_\Sigma)g^{-1}\subset \overline{RG(K)}$. En particulier, si $T$ est $R$-trivial (par ex. un tore déployé), alors on a $gRT(K_\Sigma)g^{-1}\subset \overline{RG(K)}$.
\end{prop}

\begin{proof}
    Soit $T$ un $K$-tore de $G$. On sait déjà d'après le lemme \ref{ToreDensité} que \linebreak $RT(K_\Sigma)\subset \overline{T(K)}$. Comme $T(K_\Sigma)$ est fermé dans $G(K_\Sigma)$, la notation $\overline{RT(K)}$ désigne le même objet, que l'on se place dans $T(K_\Sigma)$ ou bien dans $G(K_\Sigma)$. Comme on a évidemment $RT(K)\subset RG(K)$, on a $\overline{RT(K)}\subset \overline{RG(K)}$. D'où $RT(K_\Sigma)\subset \overline{RG(K)}$.

    Prenons désormais $T$ un $K_\Sigma$-tore de $G_{K_\Sigma}$. On écrit $T=\prod_{v\in \Sigma}T_v$ tel que pour tout $v\in \Sigma$, $T_v$ est un $K_v$-tore. Prenons $v\in \Sigma$. Par approximation faible des tores (cf. par ex. la preuve de \cite[Lemma 2.]{Gu}), il existe $g_v\in G(K_v)$ tel que $T'_v:=g_vT_vg_v^{-1}$ soit défini sur $K$.
    
    Observons alors, d'après le début de la preuve, les inclusions suivantes :
    $$g_vRT_v(K_v)g_v^{-1}=RT'_v(K_v)\subset RT'_v(K_\Sigma) \subset \overline{RG(K)}.$$
    D'où $gRT(K_\Sigma)g^{-1}\subset \overline{RG(K)}$ en posant $g=(g_v)_{v\in \Sigma}$.
\end{proof} 

\begin{rmq}
    On ignore en général si $\overline{RG(K)}$ (ou $\overline{G(K)}$) est un sous-groupe distingué de $G(K_\Sigma)$.
\end{rmq}

Terminons cette partie avec le lemme général suivant. 

\begin{lem}\label{Adherence}
    Soit $H$ un groupe topologique, $E$ une partie de $H$ et $U$ un sous-groupe ouvert de $H$.
    \begin{enumerate}
        \item L'ensemble $EU:=\{eu \mid (e,u)\in E\times U\}$ est ouvert et fermé dans $H$.
        \item On a $EU=\overline{E}U$, où $\overline{E}$ est l'adhérence de $E$ dans $H$.
    \end{enumerate}
\end{lem}

\begin{proof}
D'après \cite[Chapitre III, \S, 5., Proposition 14.]{Bourbaki3_3-4}, $H/U$ vu en tant qu'espace topologique homogène est discret. Notons $p:H\rightarrow H/U$ la projection. En particulier, $p(E)$ est ouvert et fermé dans $H/U$. Par conséquent, $EU=p^{-1}(p(E))$ est ouvert et fermé dans $H$ par continuité de $p$.

Le second point provient du précédent. En effet, on a alors : $EU\subset \overline{E}U\subset \overline{EU} = EU$.
\end{proof}

Comme cela est vu en partie \ref{conclusion}, ce lemme nous permet de faire le pont entre $\overline{G(K)}$ et le double quotient obtenu par les méthodes de recollements.

\section{Résultats principaux}\label{conclusion}

Reprenons maintenant le contexte de l'énoncé \ref{QuestionBFF}, c'est-à-dire $R$ semi-local de Dedekind, $K$ son corps de fractions, $\Rt$ et $\Kt$, etc. Par ailleurs, toutes les définitions de \cite{Article1} se généralisent à $\Kt=\prod_{\mm\in \Specm(R)}\Kt_{\mm}$ en considérant tout facteur par facteur.

Rappelons notamment qu'un sous-groupe $H$ de $G(\Kt)$ est dit global s'il est ouvert et qu'il contient $G(\Kt)^+$. Un tel groupe est dit également conforme si son action sur $\ImmBT(G_{\Kt})$ préserve les types des facettes (cf. \cite[Définition 2.1.]{Article1}). 
\medskip

Commençons par récolter des informations relatives au double quotient.

\begin{lem}\label{DoubleQuotientSimplifie}
    Soit $\Gc$ un schéma en groupes localement de présentation finie et séparé sur $\Rt$. Supposons que $G:=\Gc_{\Kt}$ soit réductif. Prenons $H$ un sous-groupe global de $G(\Kt)$, un appartement $\Ac$ de $\ImmBT(G_{\Kt})$ et $\Cc$, une chambre dans $\Ac$. Supposons que $H_{(\Ac,\Cc)}\subset \Gc(\Rt)$.
    \begin{enumerate}    
        \item On a $H\subset G(\Kt)^+\,\Gc(\Rt)$.
        \item  Si de plus $\Gc$ est défini sur $R$, et donc $G$ sur $K$, pour $g\in G(\Kt)$, on a également $G(K)\,g\, \Gc(\Rt)=\overline{G(K)}\,g\:\Gc(\Rt)$, $g\,G(\Kt)^+\,\Gc(\Rt)\subset G(K)\,g\:\Gc(\Rt)$, et $g H \subset G(K)\,g\:\Gc(\Rt)$. 
    \end{enumerate}
\end{lem}

\begin{proof}
    ${}$
    \begin{enumerate}
        \item Notons que $G(\Kt)^+\,\Gc(\Rt)$ est un sous-groupe de $G(\Kt)$ puisque $G(\Kt)^+$ est distingué dans $G(\Kt)$ (cf. \cite[6.1.]{HomAbstraits}).
        On a donc $H=G(\Kt)^+\,H_{(\Ac,\Cc)}\subset G(\Kt)^+\,\Gc(\Rt)$ d'après \cite[Lemme 2.8.]{Article1}.

        \item D'après \ref{Adherence}, on a $G(K)\,g\:\Gc(\Rt)=\overline{G(K)g}\:\Gc(\Rt)=\overline{G(K)}\:g\:\Gc(\Rt)$. En effet, il suffit de voir que $\Gc(\Rt)$ est un sous-groupe ouvert de $G(\Kt)$. C'est bien le cas car les $\Gc(\Rt_{\mm})$ sont ouverts dans les $G(\Kt_{\mm})$ d'après \cite[3.5.1 Lemme.]{GMBAdeles}.

        Ceci étant, on peut utiliser le théorème \ref{AdherenceRG} qui nous dit que $G(\Kt)^+\subset \overline{G(K)}$. En particulier, $$g\:G(\Kt)^+\,\Gc(\Rt) =G(\Kt)^+\:g\:\Gc(\Rt) \subset \overline{G(K)}\:g\:\Gc(\Rt)=G(K)\:g\:\Gc(\Rt).$$
        D'où le résultat d'après la première partie du lemme.
    \end{enumerate}
\end{proof}

On en déduit donc :

\begin{prop}\label{GroupeDeClasses}
    Reprenons le contexte du lemme précédent ($\Gc$ supposé défini sur $R$). \linebreak Lorsque cela a du sens (par exemple quand $G(K)\Gc(\Rt)$ est un sous-groupe de $G(\Kt)$), on note $c'(\Gc):=G(\Kt)/G(K)\,\Gc(\Rt)$. Supposons que $D(Z(\Kt))\subset H$, où $Z$ est un sous-groupe de Levi de $G_{\Kt}$.
    \begin{enumerate}
        \item $H$ et $G(K)\,H$ sont des sous-groupes distingués de $G(\Kt)$ de quotient abélien.\\ On note $c_H(G):=G(\Kt)/G(K)\,H$ le quotient de $G(\Kt)$ par $G(K)\,H$.
        \item Si $H_{(\Ac,\Cc)} \subset \Gc(\Rt)$ (resp. $H = G(\Kt)^+\,\Gc(\Rt)$), alors $G(K)\,\Gc(\Rt)$ est un sous-groupe distingué de $G(\Kt)$ contenant $G(K)\,H$ (resp. est égal à $G(K)\,H$) et de quotient abélien. D'où une flèche surjective (resp. bijective) $c_H(G)\rightarrow c'(\Gc)$.
        De plus, on a une bijection canonique : $$c(\Gc):=G(K)\backslash G(\Kt)/\Gc(\Rt) \overset{\sim}{\rightarrow} G(\Kt)/G(K)\,\Gc(\Rt)=:c'(\Gc).$$
    \end{enumerate}
\end{prop}

\begin{proof}
${}$
    \begin{enumerate}
        \item D'après \cite[Lemme 1.5.]{Article1}, on a $D(G(\Kt))=G(\Kt)^+\,D(Z(\Kt))$. Par conséquent, $D(G(\Kt))=G(\Kt)^+\,D(Z(\Kt))\subset H \subset G(K)\,H$. Comme l'image de $H$ et de $G(K)H$ dans $G(\Kt)^{\mathrm{ab}}$ sont des groupes distingués (car abéliens), il en est de même pour $H$ et $G(K)\,H$, et leurs quotients par $G(\Kt)$ sont bien sûr abéliens.

        \item  D'après le lemme \ref{DoubleQuotientSimplifie}, on a $H\subset G(\Kt)^+\,\Gc(\Rt)$. On observe alors :
        $$D(G(\Kt))\subset H \subset G(\Kt)^+\,\Gc(\Rt)\subset \overline{G(K)}\,\Gc(\Rt) = G(K)\,\Gc(\Rt).$$ 
        On conclut alors comme précédemment.
        \medskip
        
        Pour le dernier point, il suffit d'observer que, étant donné $g\in G(\Kt)$, on a :
        $$\overline{G(K)} \left (g\,G(\Kt)^+\,\,\Gc(\Rt)\right) = \overline{G(K)} \left(G(\Kt)^+\,g\,\Gc(\Rt)\right) = \overline{G(K)}\,G(\Kt)^+\,g\,\Gc(\Rt) = \overline{G(K)}\,g\,\Gc(\Rt).$$

        D'où finalement :
        $$\resizebox{\columnwidth}{!}{$\displaystyle
        G(K)\,g\,\Gc(\Rt)=\overline{G(K)}\,g\,\Gc(\Rt) = \overline{G(K)}\,g\left(\Gc(\Rt)\,\,\,G(\Kt)^+\right) = \left(\Gc(\Rt)\,\,G(\Kt)^+\right)\overline{G(K)}\,g=\Gc(\Rt)\overline{G(K)}\,g=\Gc(\Rt)\,G(K)\,g$}$$
        puisque $G(\Kt)^+\Gc(\Rt)$ est un sous-groupe distingué de $G(\Kt)$ (car contient $D(G(\Kt))$) et $G(K)\,g\,\Gc(\Rt)=\overline{G(K)}\,g\,\Gc(\Rt)$ d'après le lemme \ref{DoubleQuotientSimplifie}.
    \end{enumerate}
\end{proof}

\begin{lem}\label{MajorationDoubleQuotient}
    Soient $G$ un $\Kt$-groupe réductif, $S$ un $\Kt$-tore déployé de $G$ et $Z:=Z_{G_{\Kt}}(S)$ \linebreak ($Z$ désigne donc cette fois un sous-groupe de Levi de $G$ non nécessairement minimal). \linebreak
    Notons $p:Z(\Kt)\rightarrow (Z/S)(\Kt)$ la projection canonique.
    Prenons également $H$ un sous-groupe global de $G(\Kt)$ tel que $D(Z(\Kt))\subset H$.

    \begin{enumerate}
        \item Le sous-groupe $H\cap Z(\Kt)$ est global dans $Z(\Kt)$.\\
        De plus, $p$ est ouvert et $p(H\cap Z(\Kt))$ est global dans $(Z/S)(\Kt)$.
        \item On a : $$(Z/S)(\Kt)/p(H\cap Z(\Kt)) \overset{\sim}{\leftarrow} Z(\Kt)/S(\Kt)\,(H\cap Z(\Kt)) \twoheadrightarrow c_H(G).$$
        \item Si de plus $S$ est défini sur $K$, alors $Z$ également, et on a : 
        \begin{alignat*}{2}
            c_{p(H\cap Z(\Kt))}(Z/S) & = (Z/S)(\Kt)/(p(H\cap Z(\Kt))\,(Z/S)(K))\\
            & \overset{\sim}{\leftarrow} Z(\Kt)/((H\cap Z(\Kt))\,Z(K)\,S(\Kt)) = c_{H\cap Z(\Kt)}(Z)\twoheadrightarrow c_H(G).
        \end{alignat*}
    \end{enumerate}
\end{lem}

\begin{proof}
${}$
    \begin{enumerate}
        \item Comme $Z(\Kt)$ et $(Z/S)(\Kt)$ n'ont pas de sous-groupes de racines (car n'ont pas de cocaractères non centraux), on a $Z(\Kt)^+$ et $(Z/S)(\Kt)^+$ qui sont triviaux. Par conséquent, un sous-groupe de $Z(\Kt)$ (ou $(Z/S)(\Kt)$) est global si et seulement s'il est ouvert.
         
        Le sous-groupe $H\cap Z(\Kt)$ est bien sûr ouvert dans $Z(\Kt)$ puisque ce dernier est muni de la topologie induite par celle de $G(\Kt)$ et $H$ est ouvert dans $G(\Kt)$.

        Par ailleurs, $p$ est ouvert d'après \cite[3.1.2 Lemme]{GGMB} puisque $Z\rightarrow Z/S$ est lisse (car son noyau $S$ est lisse). Par conséquent, $p(H\cap Z(\Kt))$ est ouvert dans $(Z/S)(\Kt)$.

        \item Le théorème 90 de Hilbert montre que $(Z/S)(\Kt)=Z(\Kt)/S(\Kt)$. L'isomorphisme est donc une conséquence du troisième théorème d'isomorphisme (\cite[\S4.6.Théorème 4.b)]{BourbakiAlgebre1}).

        Par ailleurs, d'après la proposition \ref{ToresRG}, on a $S(\Kt)\subset \overline{G(K)}\,H=G(K)\,H$. Donc $(H\cap Z(\Kt)\,S(\Kt)\subset G(K)\,H$. La surjectivité $Z(\Kt)\twoheadrightarrow c_H(G)$ vient du fait que $G(\Kt)=G(\Kt)^+\,Z(\Kt)$ et que $G(\Kt)^+\subset H$. Il suffit donc de quotienter par \linebreak $(H\cap Z(\Kt))\,S(\Kt)$ pour obtenir la flèche surjective voulue.

        \item Enfin, le théorème 90 de Hilbert et le troisième théorème d'isomorphisme donne encore l'isomorphisme. Il suffit ensuite de voir que $$S(\Kt)\subset (H\cap Z(\Kt))\,\overline{Z(K)} = (H\cap Z(\Kt))\,Z(K)$$
        d'après le lemme \ref{Adherence} et la proposition \ref{ToresRG} pour obtenir l'égalité \linebreak $Z(\Kt)/((H\cap Z(\Kt))\,Z(K)\,S(\Kt)) = c_{H\cap Z(\Kt)}(Z)$. Enfin, la flèche surjective se construit comme précédemment en observant que $(H\cap Z(\Kt))\,Z(K)\subset H\,G(K)$.
    \end{enumerate}
\end{proof}

Faisons ensuite le pont entre la cohomologie galoisienne et la cohomologie étale par le lemme simple suivant :

\begin{lem}\label{GaloisToEtale}
    Soit $\Gc$ un schéma en groupes lisse sur $\Rt$. Notons $G:=\Gc_{\Kt}$. Rappelons que $\Galnr:= \prod_{\mm\in \Specm(R)}\Galnr_{\mm}$ agit naturellement sur $G(\Kt)=\prod_{\mm\in \Specm(R)}G(\Kt_{\mm})$. On a :
    $$\Ker\left ( H^1_{\textrm{\upshape ét}}(\Rt,\Gc)\rightarrow H^1_{\textrm{\upshape ét}}(\Kt,G) \right ) = \Ker \left ( H^1(\Galnr,\Gc(\Rtnr)) \rightarrow H^1(\Galnr,G(\Ktnr)) \right).$$
\end{lem}

\begin{proof}
    De toute évidence, l'égalité peut se montrer facteur par facteur. Autrement dit, on peut se ramener au cas où $\Rt$ est un anneau de valuation discrète hensélien.
    Observons que  $H^1(\Gamma,G(\Kt^s))=H^1_{\text{ét}}(\Kt,G)$ d'après \cite[VIII, Corollaire 2.3.]{SGA4} (où $\Gamma$ désigne le groupe de Galois absolu de $\Kt$). Par ailleurs, $H^1(\Galnr,\Gc(\Rtnr))$ vaut $H^1_{\text{ét}}(\Rt,\Gc)$. Cela est une conséquence de \cite[2.9.2.(2)]{GilleSemisimpleSchémas} et du fait que $H^1_{\text{ét}}(\Rtnr,\Gc)=1$ puisque $H^1_{\text{ét}}(\Rtnr,\Gc)\cong H^1_{\text{ét}}(\kappa^s,\Gc)$ d'après \sga{XXIV, Proposition 8.1.}.

    Observons ensuite que l'application naturelle $H^1(\Galnr,\Gc(\Rtnr)) \rightarrow H^1(\Gamma,G(\Kt^s))$ se factorise par $H^1(\Galnr,G(\Ktnr))$. D'après la suite exacte inflation-restriction (\cite[I.\S5.8.a)]{CohoGalois}), on a l'injection $H^1(\Galnr,G(\Ktnr))\rightarrow H^1(\Gamma,G(\Kt^\mathrm{s}))$. Ceci permet de conclure.
\end{proof}

Dans la preuve précédente, on a également montré que $H^1(\Galnr,\Gc(\Rtnr))=H^1_{\text{ét}}(\Rt,\Gc)$. En fait, tout $\Rt$-torseur sur $\Gc$ provient d'un unique cocycle dans $Z^1(\Galnr,\Gc(\Rtnr))$ (cf. \cite[Lemme 2.2.1.]{GilleSemisimpleSchémas} et \cite[2.9. Calculs galoisiens.]{GilleSemisimpleSchémas}). Comme dans la fin de la section \cite[4.]{Article1}, on définit alors le tordu ${}^z\Gc$ de $\Gc$ par un cocycle $z\in Z^1(\Galnr,\Gc(\Rtnr))$ comme étant le tordu au travers du torseur que $z$ définit.
\medskip

On peut enfin en déduire le théorème suivant :

\begin{thm}\label{ThmStabilisateur}
    Soit $G$ un groupe réductif sur $K$ tel que $G(\Ktnr)$ est conforme. L'application suivante est injective :
    $$H^1_{\normalfont \text{ét}}(R,\Gc) \rightarrow H^1_{\normalfont \text{ét}}(K,G)$$

    \noindent où $\Gc$ est un $R$-schéma en groupes stabilisateur d'une facette de $G$.
\end{thm}

\begin{proof}
    D'après le théorème \ref{DecoupagePb}, il suffit de montrer que la classe des groupes de Bruhat-Tits étudiée est stable par torsion intérieure, que le double quotient est trivial pour tout élément de cette classe, et que la trivialité du noyau de l'application est réalisée sur $\Rt$.
    \medskip

    Regardons la stabilité. Pour $\Xf$, un $\Gc$-torseur sur $R$, le tordu $\Gc^\Xf$ est un schéma en groupes de fibre générique $G^{\Xf_K}$. Ce dernier groupe est d'ailleurs isomorphe sur $\Kt^{\mathrm{n.r.}}$ à $G_{\Kt^{\mathrm{n.r.}}}$. En particulier, $(G^{\Xf_K})(\Ktnr)\cong G(\Ktnr)$ est conforme. Considérons maintenant $\Xf_{\Rt}$. Il provient d'un unique cocycle $z\in Z^1(\Galnr,\Gc(\Rtnr))$. D'après \cite[Proposition 4.14.(3)]{Article1}, le tordu ${}^z \Gc_{\Rt}$ est un schéma en groupes stabilisateur d'une facette de $(\Gc^\Xf)_{\Kt}$. Par définition, $\Gc$ l'est donc. Ceci conclut la stabilité.
    \medskip
    
    Le double quotient est trivial. En effet, prenons $\Cc$, une chambre de $\ImmBT(G_{\Kt})$ contenant la facette de l'énoncé du théorème. Comme $G(\Ktnr)$ est conforme, $G(\Kt)$ est donc $\Kt$-conforme. Par conséquent, $G(\Kt)_{\Cc}$ fixe $\Cc$ et donc la facette. Donc $G(\Kt)_{\Cc} \subset \Gc(\Rt)$. D'après le lemme \ref{DoubleQuotientSimplifie}, $G(\Kt)\subset G(K)\Gc(\Rt)$.
    \medskip

    Enfin, montrons que la trivialité du noyau est réalisée sur $\Rt$. D'après le lemme \ref{GaloisToEtale}, l'énoncé sous forme de cohomologie étale est équivalent à un énoncé sous forme de cohomologie galoisienne sur $\Galnr$. Le résultat découle alors de \cite[Corollaire 4.7.]{Article1} puisque $G(\Ktnr)$ est conforme.
\end{proof}

Du théorème \ref{ThmStabilisateur}, on en déduit :

\begin{coro}\label{ThmInjSSSC}
    Soit $G$ un groupe semi-simple simplement connexe quasi-déployé sur $\Ktnr$ (c'est-à-dire, tel que, pour tout $\mm\in \Spec(R)$, $G_{\Kt^{\mathrm{n.r.}}_{\mm}}$ est quasi-déployé). L'application suivante est injective :
    $$H^1_{\normalfont \text{ét}}(R,\Gc) \rightarrow H^1_{\normalfont \text{ét}}(K,G)$$

    \noindent où $\Gc$ est un schéma en groupes stabilisateur d'une facette de $G$ (ou encore parahorique de $G$, les deux coïncident).
\end{coro}

\begin{proof}
    D'après le théorème \ref{ThmStabilisateur}, il suffit de voir que $G(\Ktnr)$ est conforme. C'est une conséquence de \cite[5.2.10.(i)]{BT2}. Le résultat \cite[5.2.9.]{BT2} assure également que $\Gc$ est à fibres connexes.
\end{proof}

\pagebreak

Cela permet d'en déduire le théorème \ref{CasParfaitSSSC} :

\begin{proof}[{Démonstration du théorème \ref{CasParfaitSSSC}}]
    D'après le corollaire précédent, il suffit de montrer que tout groupe semi-simple simplement connexe sur un corps valué hensélien à corps résiduel parfait est quasi-déployé sur l'extension maximale non ramifiée. C'est en effet le cas d'après \cite[5.1.1.]{BT2}.
\end{proof}

\begin{rmq}
    Il est raisonnable de se demander si le corollaire s'étend à des groupes semi-simples simplement connexes non quasi-déployés sur $\Ktnr$. Cela est une question délicate qui nécessite une étude séparée qui va être réalisée dans un article ultérieur.
\end{rmq}

Ce résultat donne en particulier le cas semi-simple simplement connexe du théorème de Nisnevich-Guo. Avec un peu plus d'efforts, on peut également prouver le cas réductif :

\begin{thm}
    Soit $G$ un groupe réductif sur $R$. L'application suivante est injective :
    $$H^1_{\normalfont \text{ét}}(R,G) \rightarrow H^1_{\normalfont \text{ét}}(K,G).$$
\end{thm}

\begin{proof}
    Comme précédemment, grâce au théorème \ref{DecoupagePb}, il suffit de montrer que la classe des groupes réductifs sur $R$ est stable par torsion intérieure, que le double quotient est trivial pour tout élément de cette classe, et que la trivialité du noyau de l'application est réalisée sur $\Rt$.
    \medskip

    La stabilité vient du fait qu'être réductif est local pour la topologie étale (cf. \sga{Exp. XIX, Définition 2.7.}).
    \medskip

    Montrons maintenant que le double-quotient est trivial. Comme $G$ est réductif sur $R$, on peut prendre un $R$-tore déployé maximal $\mathcal{S}$ (de fibre générique $S$) et $\mathcal{Z}$ (de fibre générique $Z$) le sous-groupe de Levi minimal associé (et $Z$ est un sous-groupe de Levi minimal de $G_K$ pour $S$). Le quotient $\mathcal{Z}/\mathcal{S}$ est alors un modèle réductif de $Z/S$ et $$(\mathcal{Z}/\mathcal{S})(\Rt)=(\mathcal{Z}/\mathcal{S})(\Kt)=(Z/S)(\Kt)$$ d'après \cite[Lemme 5.1.]{Article1}, de telle sorte que Hilbert $90$ donne que $\mathcal{Z}(\Rt)\,S(\Kt)=Z(\Kt)$. Or, \cite[Lemme 5.1.]{Article1} donne que $\mathcal{Z}(\Rt)=Z(\Kt)^1$. Donc : $$D(Z(\Kt))\subset D(Z)(\Kt) \subset Z(\Kt)^1 =\mathcal{Z}(\Rt) \subset G(\Rt)\subset G(\Kt)^+\,G(\Rt).$$ D'après le lemme \ref{MajorationDoubleQuotient} et la décomposition obtenue, on a donc $G(\Kt)= G(\Kt)^+\,G(\Rt)$. Mais le lemme \ref{DoubleQuotientSimplifie} donne que $G(\Kt)^+\,G(\Rt)\subset G(K) \,G(\Rt)$. D'où finalement $G(\Kt)=G(K)\,G(\Rt)$ comme voulu.
    \medskip

    Enfin, montrons le cas hensélien. D'après \cite[Lemme 5.1.]{Article1}, on peut revenir à l'étude de points hyperspéciaux. D'après le lemme \ref{GaloisToEtale}, comme précédemment, on se ramène à un énoncé sous forme de cohomologie galoisienne sur $\Galnr$. Le résultat est alors une conséquence de \cite[Proposition 5.4.]{Article1}.
\end{proof}

\pagebreak

Le cas particulier où la facette est une chambre donne également un résultat positif :

\begin{thm}\label{ThmChambres}
    Soit $G$ un groupe réductif sur $K$. Soit $\Gc$, un $R$-schéma en groupes stabilisateur d'une chambre de $G$. Considérons l'application suivante :
    $$H^1_{\normalfont \text{ét}}(R,\Gc) \rightarrow H^1_{\normalfont \text{ét}}(K,G).$$
    \begin{enumerate}
        \item L'application est de noyau trivial.
        \item Si de plus $G_{\Kt}$ est résiduellement quasi-déployé, alors l'application est injective.
    \end{enumerate}
\end{thm}

\begin{proof}
    ${}$
    \begin{enumerate}
        \item D'après la proposition \ref{CritereEgaliteKer}, il faut montrer que tordre $\mathcal{G}$ par un élément d'une classe dans $\Ker \left( H^1_{\normalfont \text{ét}}(R,\Gc) \rightarrow H^1_{\normalfont \text{ét}}(K,G) \right)$ donne toujours un schéma en groupes stabilisateur d'une chambre, puis que la trivialité du double quotient et du noyau sur $\Rt$ est réalisée pour ces groupes.
        \medskip

        La trivialité du noyau sur $\Rt$ se ramène encore à de la cohomologie galoisienne sur $\Galnr$ d'après le lemme \ref{GaloisToEtale}. On conclut alors en utilisant \cite[Corollaire 4.8.]{Article1}.
        \medskip

        De ceci, on en déduit également la stabilité, car être un schéma en groupes stabilisateur d'une chambre se vérifie sur $\Rt$. Or, par trivialité sur noyau sur $\Rt$, tout tordu de $\Gc$ par un élément de $\Ker \left( H^1_{\normalfont \text{ét}}(R,\Gc) \rightarrow H^1_{\normalfont \text{ét}}(K,G) \right)$ est isomorphe sur $\Rt$ à $\Gc_{\Rt}$.
        \medskip
        
        Enfin, le double quotient $\Gc(\Rt)\backslash G(\Kt) / G(K)$ est trivial d'après le lemme \ref{DoubleQuotientSimplifie}. En effet, comme $G(\Kt)_{(\Ac,\Cc)}\subset G(\Kt)_{\Cc}=\Gc(\Rt)$ le point (2) donne $G(\Kt)\subset G(K) \Gc(\Rt)$.
        
        \item Toujours grâce au théorème \ref{DecoupagePb}, il suffit de montrer que la classe des groupes sur $R$ que l'on considère est stable par torsion intérieure, que le double quotient est trivial pour tout élément de cette classe, et que la trivialité du noyau de l'application est réalisée sur $\Rt$.
        \medskip

        Considérons $\Cc$, la $\Kt$-chambre associée à $\Gc$ et la $\Galnr$-chambre $\widetilde{\Cc}$ correspondante. Prenons $\Xf$, un $\Gc$-torseur sur $R$ et un cocycle $z \in Z^1(\Galnr,G(\Ktnr)_{\widetilde{\Cc}})$ correspondant à $\Xf_{\Rt}$. Le point crucial à observer est que tordre par $z$ rend toujours $\Galnr$-invariant $\widetilde{\Cc}$ d'après \cite[Proposition 4.13.(1)]{Article1}. Puisque $G_{\Kt}$ est résiduellement quasi-déployé, $\widetilde{\Cc}$ est une $\Ktnr$-chambre. Cela signifie donc que ${}^z G_{\Kt}$ est également résiduellement quasi-déployé, que ${}^z\Cc$ est une $\Kt$-chambre de $\ImmBT({}^z G_{\Kt})$, et que ${}^z \Gc_{\Rt}$ est un schéma en groupes stabilisateur de ${}^z\Cc$ d'après \cite[Proposition 4.14.]{Article1}. Il en est donc de même pour $\Gc^\Xf$.
        \medskip
    
        La trivialité du double quotient et l'injectivité dans le cas de $\Rt$ se fait alors comme précédemment.
    \end{enumerate}
\end{proof}

Terminons enfin cet article en calculant de manière exacte le noyau :
$$\Ker \left(H^1_{\normalfont \text{ét}}(R,\Gc) \rightarrow H^1_{\normalfont \text{ét}}(K,G) \right)$$

\noindent pour les $K$-groupes $G$ semi-simples adjoints et quasi-déployés sur $\Kt$, et où $\Gc$ est un schéma en groupes stabilisateur d'une facette de $G$ ou un schéma en groupes parahorique de $G$. 
\medskip

\pagebreak

Commençons donc par montrer que le double quotient $\Gc(\Rt)\backslash G(\Kt)/G(K)$ est trivial :

\begin{lem}
    Soit $G$ un groupe semi-simple adjoint sur $K$ et quasi-déployé sur $\Kt$. \linebreak On a $G(\Kt)=\overline{RG(K)}$. En particulier, $G(\Kt)=G(K)\,\Gc(\Rt)$ pour n'importe quel schéma en groupes $\Gc$ localement de présentation finie et séparé sur $R$ tel que $\Gc_K = G$.
\end{lem}

\begin{proof}
    D'après le théorème \ref{AdherenceRG} et la proposition \ref{ToresRG}, on a d'une part \linebreak $G(\Kt)^+ \subset \overline{RG(K)}$, et d'autre part l'existence d'un $g\in G(\Kt)$ tel que $g\,T(\Kt)\,g^{-1}\subset \overline{RG(K)}$, où $T$ est le centralisateur d'un $\Kt$-tore déployé maximal de $G_{\Kt}$. En effet, puisque $G$ est adjoint et quasi-déployé sur $\Kt$, le centralisateur $T$ est un tore induit, et donc $R$-trivial (cf. \cite[4.4.16. Proposition.]{BT2}). On conclut alors que $$G(\Kt) = g\,G(\Kt)^+\,T(\Kt)\,g^{-1} = G(\Kt)^+\,g\,T(\Kt)\,g^{-1}\subset \overline{RG(K)}$$ grâce à \cite[6.11.(i) Proposition.]{HomAbstraits}.

    Utilisons ensuite le point (2) de \ref{DoubleQuotientSimplifie} pour en déduire $G(\Kt)=\overline{G(K)}\,\Gc(\Rt)=G(K)\,\Gc(\Rt)$.
\end{proof}

Ceci nous permet donc de se ramener immédiatement au cas hensélien, et donc à de la cohomologie galoisienne :

\begin{coro}\label{ReductionCasHenselien}
    Reprenons les notations du lemme précédent et supposons de plus que $\Gc$ est lisse. On a alors :
    \begin{align*}
        \Ker \left(H^1_{\normalfont \text{ét}}(R,\Gc) \rightarrow H^1_{\normalfont \text{ét}}(K,G) \right)&=\Ker \left(H^1_{\normalfont \text{ét}}(\Rt,\Gc) \rightarrow H^1_{\normalfont \text{ét}}(\Kt,G) \right)\\
        &= \Ker \left(H^1(\Galnr,\Gc(\Rtnr))\rightarrow H^1(\Galnr,G(\Ktnr))\right).
    \end{align*}
\end{coro}

\begin{proof}
    Utilisons la proposition \ref{CritereEgaliteKer}. Il suffit de montrer que le double quotient associé à un tordu de $\mathcal{G}$ par un élément d'une classe dans $\Ker \left(H^1_{\normalfont \text{ét}}(R,\Gc) \rightarrow H^1_{\normalfont \text{ét}}(K,G) \right)$ est trivial. Cela est en fait évident d'après le lemme précédent car un tel tordu a une fibre générique isomorphe à $G$, par trivialité de l'élément dans $H^1_{\normalfont \text{ét}}(K,G)$. 
    
    La seconde égalité est ensuite une conséquence immédiate du lemme \ref{GaloisToEtale}.
\end{proof}

On peut désormais réutiliser les résultats \cite[Théorème 6.8.]{Article1} et \cite[Théorème 6.15.]{Article1} pour obtenir :

\begin{thm}\label{ThmParahoQSplit}
    Soit $G$ un groupe semi-simple adjoint sur $K$ et quasi-déployé sur $\Kt$. \linebreak On a :
    $$\Ker \left(H^1_{\normalfont \text{ét}}(R,\Gc) \rightarrow H^1_{\normalfont \text{ét}}(K,G)\right)=1$$

    \noindent où $\Gc$ est un schéma en groupes parahorique de $G$.
\end{thm}

\begin{thm}\label{ThmStabQSplit}
    Soit $G$ un groupe semi-simple adjoint sur $K$ et quasi-déployé sur $\Kt$. \linebreak Soit également $\Gc$, un schéma en groupes stabilisateur d'une facette de $G$. Le noyau :
    $$\Ker \left(H^1_{\normalfont \text{ét}}(R,\Gc) \rightarrow H^1_{\normalfont \text{ét}}(K,G)\right)$$
    est de cardinal $2^{\sum_{\mm\in \Specm(R)}k_\mm}$ où, pour tout $\mm\in \Specm(R)$, l'entier $k_\mm$ est majoré par le nombre de facteurs restriction de Weil d'un groupe absolument presque simple de type ${}^2D_n$ (pour $n\geq 4$) ou ${}^2A_{4n+3}$ (pour $n\geq 0$) déployé par une extension non ramifiée dans $G_{\Kt_\mm}$.
\end{thm}

\begin{rmq}
    Bien entendu, il est possible de calculer explicitement ce noyau en se ramenant à $\Kt$ grâce au corollaire \ref{ReductionCasHenselien} puis en se réduisant au cas absolument presque simple grâce à la compatibilité du noyau au produit et à la restriction de Weil (cf. \cite[Lemme 6.9.]{Article1}) et en utilisant la table \cite[Table 2.]{Article1}.
\end{rmq}

\pagebreak

\appendix

\bibliographystyle{alpha}
\bibliography{bibliography}

\bigskip

\end{document}